\documentclass{amsart}
\usepackage{amsmath,amssymb,latexsym,cancel,rotating}
\usepackage{graphicx,amssymb,mathrsfs,amsmath,color,fancyhdr,amsthm}
\usepackage[all]{xy}
\usepackage{bm}

\newtheorem{theorem}{Theorem}[section]
\newtheorem{lemma}[theorem]{Lemma}
\newtheorem{corollary}[theorem]{Corollary}
\newtheorem{definition}[theorem]{Definition}
\newtheorem{proposition}[theorem]{Proposition}

\newtheorem{thm}{Theorem}

\def\lan{\langle}    \def\ran{\rangle}

\def\az{\alpha}
\def\bz{\beta}
\def\gz{\gamma}  
\def\dz{\delta}  \def\dt{\Delta}
\def\oz{\omega}  
\def\sz{\sigma}  
\def\lz{\lambda} 
  \def\ez{\epsilon}

\def\xz{\xi}     \def\thz{\theta}

\def\cM{{\mathcal M}}  \def\cG{{\mathcal G}}
\def\cP{{\mathcal P}}
\def\cC{{\mathcal C}} \def\cD{{\mathcal D}}    \def\cL{{\mathcal L}}
\def\cH{{\mathcal H}}   \def\cB{{\mathcal B}} \def\cS{{\mathcal S}}
\def\cZ{{\mathcal Z}}  \def\cE{{\mathcal E}}
\def\cI{{\mathcal I}}    \def\cQ{{\mathcal Q}}  \def\cR{{\mathcal R}}
    
\def\cK{{\mathcal K}}  
\def\cF{{\mathcal F}}
\def\cW{{\mathcal W}}
\def\tcF{\tilde{\mathcal F}}

\def\bbN{{\mathbb N}}  \def\bbZ{{\mathbb Z}}  \def\bbQ{{\mathbb Q}}
    \def\bbF{{\mathbb F}}
\def\bbC{{\mathbb C}}  \def\bbE{{\mathbb E}}  \def\bbP{{\mathbb P}}

\def\fkf{\mathfrak f}

         \def\bfU{{\bf U}}
   
  \def\leq{\leqslant}  \def\geq{\geqslant}

\def\lra{\longrightarrow}   
\def\lmto{\longmapsto}   
\def\ra{\rightarrow}
\def\Hom{\mathrm {Hom}}  
       
\def\Ext{\mathrm{Ext}}   
\def\dim{\mathrm{dim}}
\def\dime{\mathrm{dim}\,}
\def\End{\mathrm {End}}
\def\Aut{\mbox{\rm Aut}}
\def\IC{\mathrm {IC}}
\def\rk{\mathrm {rk}\,}
\def\udim{\mathrm{\underline{dim}}\,}
\def\mod{\mbox{\rm mod}\,}  
  
\def\ind{\mathrm{ind}\,}

\def\fkm{{\frak m}}

\def\GL{\mathrm{GL}}

\def\bfb{{\bf b}}

  \def\bfa{{\bf a}}
\def\bfc{{\bf c}}  \def\bfd{{\bf d}} \def\bfe{{\bf e}}

\def\bff{{\bf f}}
\def\bfB{{\bf B}}

\def\bfa{{\bf a}}
\def\bfone{{\bf 1}}

\def\ta{{\tilde a}}
\def\bcap{{\bigcap}}
\def\bcup{{\bigcup}}

\begin{document}

\title[Tame quivers and affine bases I]
{Tame quivers and affine bases I: a Hall algebra approach to the canonical bases}

\author[J. Xiao]{Jie Xiao}
\address{Department of Mathematics, Tsinghua University, Beijing, 100084, P. R. China}
\email{jxiao@tsinghua.edu.cn (J. Xiao)}
\author[H. Xu]{Han Xu}
\address{School of Mathematical and Statistics, Beijing Institute of Technology, Beijing, 100081, P. R. China}
\email{7520220010@bit.edu.cn (H. Xu)}
\author[M. Zhao]{Minghui Zhao}
\address{School of Science, Beijing Forestry University, Beijing, 100083, P. R. China}
\email{zhaomh@bjfu.edu.cn (M. Zhao)}

\thanks{Jie Xiao was supported by NSF of China (No. 12031007), Han Xu was supported by Tsinghua University Initiative Scientific Research Program (No. 2019Z07L01006) and Minghui Zhao was supported by NSF of China (No. 11771445).}

\date{\today}


\bibliographystyle{abbrv}

\maketitle

\begin{abstract}
In \cite{Lusztig_Affine_quivers_and_canonical_bases}, Lusztig gave an explicit construction of the affine canonical basis by simple perverse sheaves.
In this paper, we construct a bar-invariant basis by using a PBW basis arising from representations of the corresponding tame quiver.
We prove that this bar-invariant basis coincides with Lusztig's canonical basis and obtain a concrete bijection between the elements in theses two bases.
The index set of these bases is listed orderly by modules in preprojective, regular non-homogeneous, preinjective components and irreducible characters of symmetric groups.
Our results are based on the work of Lin-Xiao-Zhang (\cite{Lin_Xiao_Zhang_Representations_of_tame_quivers_and_affine_canonical_bases}) and closely related with the work of Beck-Nakajima (\cite{Beck_Nakajima_Crystal_bases_and_two-sided_cells_of_quantum_affine_algebras}). A crucial method in our construction is a generalization of that in \cite{Deng_Du_Xiao_Generic_extensions_and_canonical_bases_for_cyclic_quivers}.

\end{abstract}

\setcounter{tocdepth}{1}\tableofcontents

\section{Introduction}

Let $Q=(I,H)$ be a quiver and $\bfU$ the quantized enveloping algebra associated with the quiver $Q$. Denote by $\bff$ the $\bbQ(v)$-algebra with unit $\bfone$ generated by $\thz_i,i\in I$ subject to the quantum Serre relations.
The algebra $\bff$ is isomorphic to the positive (resp. negative) part of $\bfU$.

In \cite{Ringel_Hall_algebras_and_quantum_groups}, Ringel introduced the Ringel-Hall algebra $\cH^*(kQ)$ and prove that there is an isomorphism between the generic composition subalgebra $\cC^*$ and the algebra $\bff$.
Inspired by the work of Ringel, Lusztig (\cite{Lusztig_Canonical_bases_arising_from_quantized_enveloping_algebra}\cite{Lusztig_Quivers_perverse_sheaves_and_the_quantized_enveloping_algebras}) gave a geometric realization of $\bff$ via the category of certain semisimple perverse sheaves on the variety of representations of the corresponding quiver $Q$. The set $\bfB$ of the isomorphism classes of simple perverse sheaves gives a basis of $\bff$, which is called the canonical basis.

In the case of finite type, Lusztig (\cite{Lusztig_Canonical_bases_arising_from_quantized_enveloping_algebra}) also gave an elementary algebraic construction of the canonical basis via the Ringel-Hall realization of $\bff$.
The set of isomorphism classes of representations of the corresponding quiver gives a PBW basis of $\bff$. Via this PBW basis, Lusztig construct a bar invariant basis, which is the canonical basis.

Assume that $Q$ is a Dynkin quiver.
Denote by $\mathbb{N}^{{\rm ind}\, Q}$ the set of all functions $\bfc:{\rm ind}\, Q\rightarrow\mathbb{N}$, where ${\rm ind}\, Q$ is the set of isomorphism classes of indecomposable objects in the category ${\rm rep}\,Q$ of representations of quiver $Q$. For each $\bfc:{\rm ind}\, Q\rightarrow\mathbb{N}$, there is a representation
$M(\bfc)$, such that $\cH^*(kQ)$ is spanned by the set $$\{\langle{M(\bfc)}\rangle|\bfc:{\rm ind}\, Q\rightarrow\mathbb{N}\}$$ as $\mathbb{Q}(v)$-vector space. This set is called a PBW basis of $\cH^*(kQ)$.

For each $\bfc:{\rm ind}\,Q\rightarrow\mathbb{N}$, there exists a monomial $m_{\bfc}$ on Chevalley generators  satisfying
\begin{displaymath}
m_{\bfc}=\langle{M(\bfc)}\rangle+\sum_{\bfc'<\bfc}a^{\bfc'}_{\bfc}\langle{M(\bfc')}\rangle,
\end{displaymath}
with $a^{\bfc'}_{\bfc}\in \mathbb{Z}[v,v^{-1}]$.
Since $\overline{m_{\bfc}}=m_{\bfc}$, we have
\begin{displaymath}
\overline{\langle{M(\bfc)}\rangle}=\langle{M(\bfc)}\rangle+\sum_{\bfc'<\bfc}b^{\bfc'}_{\bfc}\langle{M(\bfc')}\rangle,
\end{displaymath}
with $b^{\bfc'}_{\bfc}\in \mathbb{Z}[v,v^{-1}]$.

According to an elementary linear algebra method of Lusztig, we can construct a bar-invariant basis
$$\{C_{\bfc}|\bfc:{\rm ind}\, Q\rightarrow\mathbb{N}\}.$$
Under the isomorphism between $\cC^*$ and $\bff$, this bar-invariant basis induces a basis of $\bff$, which is the canonical basis $\mathbf{B}$ of $\bff$.

In the case of affine type, Lusztig (\cite{Lusztig_Affine_quivers_and_canonical_bases}) gave an explicit construction of the canonical basis by simple perverse sheaves indexed by the isomorphism classes of aperiodic modules of the corresponding tame quiver and irreducible representations of symmetric groups. It is an important problem to generalize the elementary algebraic construction to the affine canonical basis.

Beck-Chari-Pressley (\cite{Beck_Chari_Pressley_An_algebraic_characterization_of_the_affine_canonical_basis}) and Beck-Nakajima (\cite{Beck_Nakajima_Crystal_bases_and_two-sided_cells_of_quantum_affine_algebras}) gave an algebraic construction of the canonical basis by using a PBW basis (see \cite{Nakajima_Crystal_canonical_and_PBW_bases_of_quantum_affine_algebras}).

Let $A=(a_{ij})_{i,j\in I}$ be a symmetric generalized Cartan matrix of affine type, where $I=\{0,1,\ldots,n\}$, $0\in I$ is the exceptional point and $I_0=I\backslash\{0\}$.

Beck-Nakajima defined an infinite sequence $$h=(\cdots,i_{-1},i_0,i_1,\cdots)$$
in $I$ such that  $R=R_{>}\bigsqcup R_{<}$ is the set of all positive real roots, where
$$R_{>}=\{\beta_0=\alpha_{i_0},\beta_{-1}=s_{i_0}(\alpha_{i_{-1}}),\beta_{-2}=s_{i_0}s_{i_{-1}}(\alpha_{i_{-2}}),\cdots\}$$
and
$$R_{<}=\{\beta_{1}=\alpha_{i_1},\beta_{2}=s_{i_1}(\alpha_{i_2}),\beta_{3}=s_{i_1}s_{i_2}(\alpha_{i_3}),\cdots\}.$$
Let $$E_{\beta_k}=\left\{\begin{array}{cc}
              T^{-1}_{i_0}T^{-1}_{i_{-1}}\cdots T^{-1}_{i_{k+1}}(E_{i_k}) & \textrm{if $k\in\mathbb{Z}_{\leq0}$},\\
              T_{i_1}T_{i_2}\cdots T_{i_{k-1}}(E_{i_k}) & \textrm{if $k\in\mathbb{Z}_{>0}$},
            \end{array}
\right.$$
where $T_j$ is the Lusztig's symmetry corresponding to $j\in I$.
Then $E_{\beta_k}$ is a real root vector for the root $\beta_k\in R$.

Also, Beck-Nakajima defined integral imaginary root vectors $\tilde{P}_{i,k}$ from real root vectors by a recursive identity.
Denote by $\mathscr{P}$ the set of all partitions.
For $\lambda=(\lambda_1\geq \lambda_2\geq\cdots)\in\mathscr{P}$, define the Schur function $$S_{\lambda}=\det(\tilde{P}_{i,\lambda_k-k+m})_{1\leq k,m\leq t},$$
where $t$ is the length of $\lambda$.

For a map $c_0:I_0\rightarrow \mathscr{P}$, denote $$S_{c_0}=\prod_{i=1}^{n}S_{c_0(i)}.$$

Let $\mathcal{E}$ be the set of $\bar{c}=(c,c_0)$, where $c_0:I_0\rightarrow \mathscr{P}$ is a map and $c:\mathbb{Z}\rightarrow\mathbb{N}$ is a function with finite support.
For any $\bar{c}\in\mathcal{E}$ and $p\in\mathbb{Z}$, let
\begin{eqnarray*}
L(\bar{c},p)&=&\left(E_{i_p}^{(c(p))}T_{i_p}^{-1}(E_{i_{p-1}}^{(c(p-1))})T_{i_p}^{-1}T_{i_{p-1}}^{-1}(E_{i_{p-2}}^{(c(p-2))})\cdots\right)\\
&\times&T_{i_{p+1}}T_{i_{p+2}}\cdots T_{i_0}(S_{c_0})\\
&\times&\left(\cdots T_{i_{p+                 1}}T_{i_{p+2}}(E_{i_{p+3}}^{(c(p+3))})T_{i_{p+1}}(E_{i_{p+2}}^{(c(p+2))})E_{i_{p+1}}^{(c(p+1))}\right).
\end{eqnarray*}

Beck-Nakajima defined a certain partial ordering $\prec_{p}$ on $\cE$ for each $p\in\bbZ$ and proved that
$$\overline{L(\bar{c},p)}=L(\bar{c},p)+\sum_{\bar{c}\prec_{p}\bar{c}'}a_{\bar{c},\bar{c}'}L(\bar{c}',p)$$ with $a_{\bar{c},\bar{c}'}\in\bbZ[v,v^{-1}]$.

According to an elementary linear algebra method of Lusztig, Beck-Nakajima constructed a bar-invariant basis $$\{b(\bar{c},p)|\bar{c}\in\mathcal{E}\}$$ and proved that
 this is the canonical basis $\mathbf{B}$ of $\bff$.

In the case of the Kronecker quiver, Chen (\cite{Chen_Root_vectors_of_the_composition_algebra_of_the_Kronecker_algebra}) and Zhang (\cite{Zhang_PBW-basis_for_the_composition_algebra_of_the_Kronecker_algebra}) constructed PBW bases of $\bff$, and McGerty (\cite{Mcgerty_The_Kronecker_quiver}) gave the interpretation of PBW bases of $\bff$ constructed by Beck-Nakajima in terms of representations of Kronecker quiver.
In the case of cyclic quiver, Deng-Du-Xiao (\cite{Deng_Du_Xiao_Generic_extensions_and_canonical_bases_for_cyclic_quivers}) gave an inductive method to construct a PBW basis from the monomial basis. By using the triangular relations between the PBW basis and the  monomial basis, they gave an algebraic construction of the canonical basis.
Based on the work of Chen, Zhang and Deng-Du-Xiao, Lin-Xiao-Zhang (\cite{Lin_Xiao_Zhang_Representations_of_tame_quivers_and_affine_canonical_bases}) gave an algebraic construction of the canonical basis by using the representations of tame quiver.

In this paper, we give an algebraic construction of the canonical basis by using a PBW basis arising from representations of the corresponding tame quiver. Our results are based on the work of Lin-Xiao-Zhang (\cite{Lin_Xiao_Zhang_Representations_of_tame_quivers_and_affine_canonical_bases}) and are closely related with the work of Beck-Nakajima (\cite{Beck_Nakajima_Crystal_bases_and_two-sided_cells_of_quantum_affine_algebras}). A crucial method in our construction is a generalization of that in \cite{Deng_Du_Xiao_Generic_extensions_and_canonical_bases_for_cyclic_quivers}.

Let $\cG'$ be the set consisting of all triples $\bfc=(\bfc_-,\bfc_0,\bfc_+)$, where $\bfc_-$ arising from preprojective components, $\bfc_+$ arising from preinjective components, and $\bfc_0$ arising from  non-homogeneous tubes of ${\rm rep}\, Q$. We define the index set $\cG$ consisting of all pairs $(\bfc,t_\lz)$ with $\bfc\in\cG'$ and $t_\lz$ be the character of the irreducible module $S_\lz$ of the symmetric group $\mathfrak{S}_{|\lz|}$ for a partition $\lz$.

Let $\cH^0_q$ be the $\bbQ(v_q)$-subalgebra of the Ringel-Hall algebra $\cH^*(kQ)$ generated by $u_i=u_{[S_i]},i\in I$ and $u_{[M]}$ for all representations $M$ in non-homogeneous tubes, where $S_i$ is the simple module concentrated at $i\in I$.
For $(\bfc,t_\lz)\in\cG$, we define an element
$N(\bfc,t_\lz)$ in $\cH^0_q$ and we prove that $$\{N(\bfc,t_\lz)|(\bfc,t_\lz)\in\cG\}$$ is a $\bbQ(v_q)$-basis of $\cH^0_q$. 
Moreover, this is a "generic" basis independent of $q$, which allows us to define a generic form $\cH^0$ and we have $\cC^*\subset\cH^0$.

Let $\cG^a$ be the subset of $\cG$ consisting of aperiodic indices. 
For $(\bfc,t_\lz)\in\cG^a$, we define a monomial basis of the composition algebra
$$\{\fkm^{\oz(\bfc,t_\lz)}|(\bfc,t_\lz)\in\cG^a\}.$$
Then we use the method in \cite{Deng_Du_Xiao_Generic_extensions_and_canonical_bases_for_cyclic_quivers} to construct a PBW basis
$$\{E(\bfc,t_\lz)|(\bfc,t_\lz)\in\cG^a\}$$ for $\cC^*_\cZ$ inductively, where $\cC^*_\cZ$ is the integral form of $\cC^*$.

At last, we construct a bar-invariant basis $$\{C(\bfc,t_\lz)|(\bfc,t_\lz)\in\cG^a\}$$ by using the triangular relation between the PBW basis and the monomial basis.

Let $B(\bfc,t_\lz)=[\IC(X(\bfc,|\lambda|),\cL_\lambda)]$, the image of perverse sheaf $\IC(X(\bfc,|\lambda|),\cL_\lambda)$ in the Grothendieck group, where the local system $\cL_\lambda$ is given by an irreducible module of $\mathfrak{S}_{|\lambda|}$. By Lusztig, the set
$$\bfB=\{B(\bfc,t_\lz)|(\bfc,t_\lz)\in\cG^a\}$$
is the canonical basis.

The following theorem is the main result in this paper.
\begin{thm}\label{Theorem:1}
For all $(\bfc,t_\lz)\in\cG^a$, it holds that $C(\bfc,t_\lz)=B(\bfc,t_\lz)$ under the isomorphism  between $\cC^*$ and $\bff$.
\end{thm}
Hence the bar-invariant basis constructed in this paper is the canonical basis $\bfB$ and we give an algebraic construction of the canonical basis.

We write few words here to explain why we need a further investigation upon the work of Lin-Xiao-Zhang (\cite{Lin_Xiao_Zhang_Representations_of_tame_quivers_and_affine_canonical_bases}) and Beck-Nakajima (\cite{Beck_Nakajima_Crystal_bases_and_two-sided_cells_of_quantum_affine_algebras}).

We hope to use the well-known structure of representations of tame quivers to construct a PBW basis of $\bff$ and show its relation with a monomial basis clearly, similarly to the case of finite type. New progresses in this paper are as follows.

\begin{enumerate}
\item We introduce the reductive form $\cH^0$ by extending the generic composition algebra $\cC^*$, which is a $\bbQ(v)$-algebra with a natural PBW basis. By using the PBW basis of $\cH^0$, we define a monomial basis and construct a PBW basis for the integral form $\cC^*_\cZ$ of $\cC^*$, and obtain a unipotent triangular relation between these two bases over $\cZ=\bbZ[v,v^{-1}]$. We prove also that these two bases are $\cZ$-bases of $\cC^*_\cZ$. This is an improvement of the main result of \cite{Lin_Xiao_Zhang_Representations_of_tame_quivers_and_affine_canonical_bases} (see Section 6.5 of \cite{Lin_Xiao_Zhang_Representations_of_tame_quivers_and_affine_canonical_bases}).
\item Each basis element of the PBW basis we constructed corresponds to an aperiodic module of the quiver $Q$ and a partition $\lambda$. This fulfill, in an algebraic way, that Lusztig's perverse sheaves $\IC(X(\bfc,|\lambda|),\cL_\lz)$ with the parameter $\bfc$ corresponds to the aperiodic representation $M(\bfc)$ of $Q$.
\item For a partition $\lambda$ of homogeneous regular modules
    of ${\rm rep}\,
    Q$, we have the permutation action of the symmetric group $\mathfrak{S}_{|\lambda|}$. The indices $t_{\lambda}$ in the PBW basis elements $E(\bfc,t_\lz)$ are given by irreducible characters of $\mathfrak{S}_{|\lambda|}$, which correspond bijectively to the local systems $\mathcal{L}_{\lambda}$ of Lusztig's perverse sheaves.
\item The final step is the construction of a bar-invariant basis according to an elementary linear algebra method of Lusztig (\cite{Lusztig_Introduction_to_quantum_groups}), by using the unipotent triangular relation between the monomial basis and the PBW basis. We prove that it coincides with the canonical basis, moreover, each element in the canonical basis is explicitly realized.
\end{enumerate}

In Section 2 and 3, we recall some basic definitions for quiver representations, Ringel-Hall algebras and quantum groups.
In Section 4, we recall the algebraic construction of the canonical basis for the quantized enveloping algebra corresponding to a cyclic quiver given by Deng-Du-Xiao.
The representations of affine quivers are recalled in Section 5.
Then we define the extended composition algebra $\cH^0$, the PBW basis elements and their index set in Section 6, 7, 8.
We give the construction of a monomial basis $\{\fkm^{\oz(\bfc,t_\lz)}|(\bfc,t_\lz)\in\cG^a\}$, a PBW basis $\{E(\bfc,t_\lz)|(\bfc,t_\lz)\in\cG^a\}$ and a bar-invariant basis $\{C(\bfc,t_\lz)|(\bfc,t_\lz)\in\cG^a\}$
in Section 9 and prove in Section 11 that the bar-invariant basis is exactly the canonical basis introduced in Section 10.
Based on the proof of Theorem \ref{Theorem:1}, we give a more direct  algorithm to compute $C(\bfc,t_\lz)$ in the last section.

\section{Quivers and representations}

\subsection{Quivers}
A quiver $Q=(I,H,s,t)$ consists of a vertex set $I$, an arrow set $H$, and two maps $s,t:H\ra I$ such that an arrow $h\in H$ starts at $s(h)$ and terminates at $t(h)$.
A quiver $Q$ is called acyclic if there is no oriented cycle.
Throughout the paper, we always assume that $Q$ is finite, i.e., $I$ and $H$ are finite sets.

\subsection{Representations}
Let $k$ be a field. A representation $(V_i,V_h)_{i\in I,h\in H}$ of $Q$ over $k$ is a family of $k$-vector spaces $V_i,i\in I$ and linear maps $V_h:V_{s(h)}\ra V_{t(h)},h\in H$.
A representation $(V_i,V_h)$ is called nilpotent if the map $V_{h_l}\circ V_{h_{l-1}}\circ \cdots \circ V_{h_1}:V_{s(h_1)}\ra V_{s(h_1)}$ is nilpotent, for any $h_1,h_2,\dots,h_l\in H$ with $t(h_j)=s(h_{j+1}),1\leq j<l,s(h_1)=t(h_l)$.
Note that all representations are nilpotent if $Q$ is acyclic.
In this paper, all modules mentioned will be nilpotent.
By $S_j$ we denote the simple module $(V_i,V_h)$ with $V_j=k$ and $V_i=0,i\neq j$.

Let $kQ$ be the path algebra of $Q$ over $k$.
By $\mod kQ$ we denote the category of finite dimensional left $kQ$-modules.
It is well-known that $\mod kQ$ is equivalent to the category $\mathrm{rep}_kQ$ of finite dimensional representations of $Q$ over $k$.
We shall simply identify $kQ$-modules with representations of $Q$ over $k$.

\subsection{Dimension vectors and the Euler form}
The set of isomorphism classes of (nilpotent) simple $kQ$-modules is naturally indexed by the vertex set $I$ of $Q$.
Then the Grothendieck group $G(kQ)$ of $\mod kQ$ is the free Abelian group $\bbZ I$. We say $\nu\leq\nu'$ for $\nu,\nu'\in\bbZ I$, if $\nu'-\nu\in\bbN I$.

For each nilpotent $kQ$-module $M=(M_i,M_h)$, the dimension vector
 $$\udim M=\sum_{i\in I}(\dim_k\,M_i)i$$
 is an element of $\bbN I$.

The Euler form $\lan -,- \ran$ on $\bbZ I$ is defined by
$$\lan\nu,\nu'\ran=\sum_{i\in I} \nu_i\nu'_i-\sum_{h\in H}\nu_{s(h)}\nu'_{t(h)}$$
for $\nu=\sum_{i\in I}\nu_i i$ and $\nu'=\sum_{i\in I}\nu'_i i$ in $\bbZ I$.
For any two nilpotent $kQ$-modules $M$ and $N$, one has
$$\lan\udim M,\udim N\ran=\dim_k\,\Hom_{kQ}(M,N)-\dim_k\,\Ext^1_{kQ}(M,N).$$
The symmetric Euler form is defined as
$$(\nu,\nu')=\lan\nu,\nu'\ran+\lan\nu',\nu\ran\,\, \mathrm{for}\,\, \nu,\nu'\in \bbZ I.$$
This gives rise to a symmetric generalized Cartan matrix $C=((i,j))_{i,j\in I}$.
It is easy to see that $C$ is independent of the field $k$ and the orientation of $Q$.

\subsection{Root systems and Weyl groups}

Since we only consider tame quivers, the root systems can be defined in an easier way (see \cite{Crawley-Boevey_William_Lectures_on_representations_of_quivers}).
A nonzero element $\nu$ in $\bbZ I$ is called a root if $0\leq(\nu,\nu)\leq 2$.
Denote by $R$ the set of all roots.
Let $R_+=R\cap\bbN I$ and $R_-=-R_+$.
It is well-known that
$$R=R_+\sqcup R_-.$$

One can check that $(\nu,\nu)\in 2\bbZ$.
Thus if $\nu$ is a root, then $(\nu,\nu)=0$ or $2$.
A root is called real (resp. imaginary) if $(\nu,\nu)=2$ (resp. 0).
Denote by $R^{\rm re}_+$ (resp. $R^{\rm im}_+$) the set of all real (resp. imaginary) positive roots.

When $Q$ is of affine type, it is well-known that all imaginary positive roots are of the form $m\dz$ where $m\in\bbZ_{>0}$ and $\dz\in R_+$.
The root $\dz=\sum_{i\in I}\dz_ii$ is called the minimal imaginary positive root. Let $|\dz|=\sum_i \dz_i$.

Given any $i\in I$, we denote by $s_i$ the automorphism of $\bbZ I$ defined by
$$s_i(\nu)=\nu-(\nu,i)i$$
for any $\nu\in\bbZ I$.
Note that $s_i$ is involutive, i.e. $s_i\circ s_i=\mathrm{id}_{\bbZ I}$.
The bilinear form $(-,-)$ is $s_i$-invariant for all $i\in I$.
The Weyl group is the subgroup of $\Aut(\bbZ I)$ generated by $s_i,i\in I$.

\section{Ringel-Hall algebras and quantum groups}
\subsection{Notations}
Given $n\in\bbZ$ and any symbol $x$, we define the notation
$$[n]_x=\frac{x^n-x^{-n}}{x-x^{-1}}.$$
We define $[0]_x=[0]_x^!=1$ and $[n]_x^!=[n]_x[n-1]_x\cdots[1]_x$ for $n\in\bbZ_{>0}$.

\subsection{Ringel-Hall algebras}
Let $k=\bbF_q$ be a finite field with $q$ elements.
Given three modules $L,M,N$ in $\mod kQ$, let $g^L_{MN}$ be the number of $kQ$-submodules $W$ of $L$ such that $W\cong N$ and $L/W\cong M$ in $\mod kQ$.
Let $v_q=\sqrt{q}\in\bbC$ and $\cM$ be the set of isomorphism classes of finite dimensional nilpotent $kQ$-modules. We denote by $[M]$ the isomorphism class of $M$ for any $kQ$-module $M$.

The (twisted) Ringel-Hall algebra $\cH^*=\cH^*(kQ)$ of $kQ$ is by definition the $\bbQ(v_q)$-space with basis $\{u_{[M]}|[M]\in\cM\}$, whose multiplication $*$ is given by
$$u_{[M]}*u_{[N]}=v_q^{\lan\udim M,\udim N\ran}\sum_{[L]\in\cM}g^L_{MN}u_{[L]}.$$

Following \cite{Ringel_PBW-bases_of_quantum_groups}, for any $kQ$-module $M$, we denote
$$\lan M\ran=v_q^{-\dime M+\dime\End_{kQ} M}u_{[M]}.$$

Note that $g^L_{MN}$ depends only on the isomorphism classes of $M,N$ and $L$, and for fixed isomorphism classes of $M,N$ there are only finitely many isomorphism classes $[L]$ with $g^L_{MN}\neq 0$.
It is clear that $\cH^*(kQ)$ is an associative $\bbQ(v_q)$-algebra with unit $1=u_0$, where $0$ denotes the isomorphism class of zero module.
We sometimes use the notation $\cH^*_q=\cH^*_q(Q)$ instead of $\cH^*(kQ)$ to indicate the dependence on $q$ when needed.

Here is a property that will be used later.
\begin{lemma}[\cite{Ringel_PBW-bases_of_quantum_groups}]\label{M^(m)}
An indecomposable $kQ$-module $M$ is called exceptional if $$\Ext^1_{kQ}(M,M)=0.$$ If $M$ is exceptional, then
$$\lan M\ran^{(m)}=\lan M^{\oplus m}\ran.$$
If $M,N\in\mod kQ$ with $\Hom_{kQ}(N,M)=0$ and $\Ext^1_{kQ}(M,N)=0$, then
$$\lan M\ran*\lan N\ran=\lan M\oplus N\ran,$$
where $x^{(m)}=x^m/[m]_{v_q}^!$.
\end{lemma}

\subsection{Quantum groups}\label{1}
Let $v$ be an indeterminant and $\cZ=\bbZ[v,v^{-1}]$.
Let $\bff$ be the $\bbQ(v)$-algebra with unit $1$ generated by $\thz_i,i\in I$ subject to the quantum Serre relations
$$\sum_{s+r=1-(i,j)}(-1)^s\thz_i^{(s)}\thz_j\thz_i^{(r)}=0, i\neq j.$$

For any $\nu=\sum_i\nu_ii\in\bbN I$, we denote by $\bff_\nu$ the $\bbQ(v)$-subspace of $\bff$ spanned by the monomials $\thz_{i_1}\thz_{i_2}\cdots\thz_{i_r}$ such that $\sum_{j}i_j=\nu$.
Then we have a direct sum decomposition $\bff=\oplus_{\nu\in\bbN I}\bff_\nu$.
We have $\bff_\nu\bff_{\nu'}\subset\bff_{\nu+\nu'}$, $1\in\bff_0$ and $\thz_i\in\bff_i$.

An element $x$ in $\bff$ is said to be homogeneous if it belongs to $\bff_\nu$ for some $\nu$.
We set $|x|=\nu$.

The integral form $\bff_\cZ$ is by definition the $\cZ$-subalgebra of $\bff$ generated by $\thz_i^{(m)},i\in I,m\geq 0$.

The quantized enveloping algebra $\bfU$ associated with a quiver $Q$ is an associative algebra over $\bbQ(v)$ with unit $\bfone$ generated by the elements $E_i,F_i,K_i^{\pm}(i\in I)$ subject to the following relations:
$$K_0=\bfone, K_\nu K_{\nu'}=K_{\nu+\nu'}\,\,\,\,\,\text{for all}\,\,\,\,\, \nu,\nu'\in \bbZ I;$$
$$K_\nu E_i K_{-\nu}=v^{(\nu,i)}E_i\,\,\,\,\,\text{for all}\,\,\,\,\, i\in I,\nu\in \bbZ I;$$
$$K_\nu F_i K_{-\nu}=v^{-(\nu,i)}F_i\,\,\,\,\,\text{for all}\,\,\,\,\, i\in I,\nu\in \bbZ I;$$
$$E_iF_j-F_jE_i=\dz_{ij}\frac{K_i-K_{-i}}{v-v^{-1}}\,\,\,\,\,\text{for all}\,\,\,\,\, i,j\in I;$$
$$\sum^{1-(i,j)}_{k=0}(-1)^k E_i^{(k)}E_j E_i^{(1-(i,j)-k)}=0\,\,\,\,\,\text{for}\,\,\,\,\, i\neq j\in I;$$
$$\sum^{1-(i,j)}_{k=0}(-1)^k F_i^{(k)}F_j F_i^{(1-(i,j)-k)}=0\,\,\,\,\,\text{for}\,\,\,\,\, i\neq j\in I.$$
Here $E_i^{(n)}=E_i^n/[n]_{v}^!$,$F_i^{(n)}=F_i^n/[n]_{v}^!$.

Let $\bfU^+$ (resp. $\bfU^-$) be the subalgebra of $\bfU$ generated by $E_i$ (resp. $F_i$) for $i\in I$, and let $\bfU^0$ be the subalgebra of $\bfU$ generated by $K_i$ for $i\in I$.
It is well known that $\bfU$ has the triangular decomposition
$$\bfU=\bfU^-\otimes\bfU^0\otimes\bfU^+.$$
It is easy to see that there exist well-defined $\bbQ(v)$-algebra monomorphisms $\bff\ra\bfU(x\mapsto x^+)$ and $\bff\ra\bfU(x\mapsto x^-)$ with image $\bfU^+$ and $\bfU^-$ respectively satisfying $E_i=\thz_i^+$ and $F_i=\thz_i^-$. We denote $\bfU^+_\cZ=\bff^+_\cZ$, $\bfU^-_\cZ=\bff^-_\cZ$.

\subsection{Composition subalgebras}

Let $\cW$ be an infinite set consisting of numbers of the form $q=p^s$ with $p$ prime and $s$ positive integer.
The integral generic composition algebra $\cC^*_\cZ=\cC^*_\cZ(Q)$ is defined to be the $\cZ$-subalgebra of $\Pi_{q\in\cW}\cH^*_q(Q)$ generated by $u_i^{(m)}=(u^{(m)}_{[S_i]})_q=(\lan S_i^{\oplus m}\ran)_q,i\in I,m\geq 0$, where $v$ acts by $(v_q)_q$ and $v^{-1}$ acts by $(v_q^{-1})_q$.
Then define the generic composition algebra $$\cC^*=\bbQ(v)\otimes_\cZ\cC^*_\cZ.$$

It is well-known (\cite{Ringel_Hall_algebras_and_quantum_groups}) that there is an isomorphism of algebras
 \begin{align}
\Phi:\bff&\lra \cC^*\notag\\
\thz_i&\lmto u_i,\notag
\end{align}
such that $\Phi(\bff_\cZ)=\cC^*_\cZ$.

Denote by $\Phi^+:\bfU^+\ra\cC^*$ the composition of $\Phi$ and the isomorphism $\bfU^+\ra\bff$.

For a finite field $\bbF_q$, we denote by $\Phi_q$ the specializing map given by the composition of $\Phi:\bff_\cZ\ra\cC^*_\cZ$, $\cC^*_\cZ\hookrightarrow\Pi_{q'\in\cW}\cH^*_{q'}(Q)$ and the projection $\Pi_{q'\in\cW}\cH^*_{q'}(Q)\ra \cH^*_q(Q)$.

\subsection{Monomials}\label{Section: def of monomial}
Let $\cS$ be the set of all pairs $\oz=(\underline{i},\underline{a})$, where $\underline{i}=(i_1,i_2,\cdots,i_m)$ is a sequence in $I$ and $\underline{a}=(a_1,a_2,\cdots,a_m)$ is a sequence in $\bbN$ of the same length.
For such a pair $\oz=(\underline{i},\underline{a})\in\cS$, we define a monomial
$$\fkm^\oz=\thz_{i_1}^{(a_1)}\thz_{i_2}^{(a_2)}\cdots\thz_{i_m}^{(a_m)}$$
in $\bff$.
By definition, $\bff_\cZ$ is $\cZ$-spanned by $\fkm^\oz$ for all $\oz\in\cS$.

For simplicity, under the isomorphism $\Phi$, we still denote $\Phi(\fkm^\oz)=u_{i_1}^{(a_1)}*u_{i_2}^{(a_2)}*\cdots* u_{i_m}^{(a_m)}$ by $\fkm^\oz$.
When specializing to a field $\bbF_q$, we still denote $\Phi_q(\fkm^\oz)=u_{[S_{i_1}]}^{(a_1)}*u_{[S_{i_2}]}^{(a_2)}*\cdots* u_{[S_{i_m}]}^{(a_m)}$ by $\fkm^\oz$.

\subsection{A bilinear form and the coproduct}\label{Section: Bilinear form and coprouct}
Define an algebra structure on $\bff\otimes\bff$ by
$$(x_1\otimes x_2)(y_1\otimes y_2)=v^{(|x_2|,|y_1|)}x_1y_1\otimes x_2y_2,$$
where $x_1,x_2,y_1,y_2$ are homogeneous.

The coproduct $r:\bff\ra\bff\otimes\bff$ is the $\bbQ(v)$-algebra homomorphism such that $r(\thz_i)=\thz_i\otimes 1+1\otimes \thz_i, i\in I$.

There is a non-degenerate symmetric bilinear form $(-,-)$ on $\bff$ with values in $\bbQ(v)$ such that $(1,1)=1$ and
\begin{enumerate}
\item[(a)] $(\thz_i,\thz_j)=\dz_{ij}(1-v^{-2})^{-1},i,j\in I$;

\item[(b)] $(x,yy')=(r(x),y\otimes y'), x,y,y'\in\bff$,
where the form on $\bff\otimes\bff$ is defined by $(x_1\otimes y_1,x_2\otimes y_2)=(x_1,x_2)(y_1,y_2)$.
\end{enumerate}

Let $k=\bbF_q$, the inner product $(-,-)$ on $\cH^*(kQ)$ is given by the formula (\cite{Green_Hall_algebras_hereditary_algebras_and_quantum_groups})
$$(\lan M\ran,\lan N\ran)=\dz_{MN}\frac{v_q^{2\dime\End (M)}}{a_M},$$
where $a_M=|\Aut (M)|$.
This inner product is also well-defined on $\cC^*$, which coincides with the bilinear form on $\bff$ under the isomorphism $\Phi$.

Following \cite{Green_Hall_algebras_hereditary_algebras_and_quantum_groups} and \cite{Ringel_Hall_algebras_and_quantum_groups}, under the homomorphism $\Phi_q$, the linear operation $r:\cH^*(kQ)\ra\cH^*(kQ)\otimes\cH^*(kQ)$ given by
$$r(u_{[L]})=\sum_{[M],[N]}v_q^{\lan \udim M,\udim N \ran}g^{L}_{M,N}\frac{a_M a_N}{a_L}u_{[M]}\otimes u_{[N]}$$
coincides with the coproduct $r:\bff\ra\bff\otimes\bff$.

\subsection{The bar involution and the canonical basis}\label{Section: The bar involution and the canonical basis}

There is a $\bbQ$-algebra automorphism $\bar{\cdot}:x\mapsto \bar{x}$ on $\bff$ defined by $\bar{\thz}_i=\thz_i$, $\bar{v}=v^{-1}$, called bar involution.

By \cite{Lusztig_Introduction_to_quantum_groups}, let $\cB$ be the set of all $x\in\bff_\cZ$ such that $\bar{x}=x$ and $(x,x)\in 1+v^{-1}\bbZ[[v^{-1}]]$. The set $\cB$ is called the signed canonical basis of $\bff$ and Lusztig defined by induction a particular subset $\bfB\subset\cB$ such that $\cB=\bfB\sqcup-\bfB$ (see \cite{Lusztig_Introduction_to_quantum_groups}, 14.4). The set  $\bfB$ is called the canonical basis of $\bff$. The aim of this paper is to construct $\bfB$ by Ringel-Hall algebra approach.
\subsection{Lusztig's symmetries}
Lusztig in \cite{Lusztig_Introduction_to_quantum_groups} introduced the symmetries $T_i=T_{i,1}'':\bfU\ra\bfU$ for $i\in I$, as algebra automorphisms of $\bfU$ defined by relations:
$$T_i(K_\nu)=K_{s_i(\nu)}; T_i(E_i)=-F_iK_i; T_i(F_i)=-K_i^{-1}E_i,$$
$$T_i(E_j)=\sum_{r+s=-(i,j)}(-1)^rv^{-r}E_i^{(s)}E_jE_i^{(r)}, j\neq i,$$
$$T_i(F_j)=\sum_{r+s=-(i,j)}(-1)^rv^{r}F_i^{(r)}F_jF_i^{(s)}, j\neq i,$$
and its inverse $T_i^{-1}=T_{i,-1}':\bfU\ra\bfU$ defined by relations:
$$T_i^{-1}(K_\nu)=K_{s_i(\nu)}; T_i^{-1}(E_i)=-K_i^{-1}F_i; T_i^{-1}(F_i)=-E_iK_i,$$
$$T_i^{-1}(E_j)=\sum_{r+s=-(i,j)}(-1)^rv^{-r}E_i^{(r)}E_jE_i^{(s)}, j\neq i,$$
$$T_i^{-1}(F_j)=\sum_{r+s=-(i,j)}(-1)^rv^{r}F_i^{(s)}F_jF_i^{(r)}, j\neq i.$$

For $i\in I$, define
$$\bfU^+[i]=\{x\in\bfU^+|T_i(x)\in\bfU^+\},{^\sz\bfU^+}[i]=\{\in\bfU^+|T_i^{-1}(x)\in\bfU^+\}.$$
Then $T_i:\bfU^+[i]\ra{^\sz\bfU^+}[i]$ is an isomorphism and $T_i^{-1}$ is its inverse. Moreover, let $\bfU^+_\cZ[i]=\bfU^+_\cZ\cap\bfU^+[i]$ and ${^\sz\bfU^+_\cZ[i]}=\bfU^+_\cZ\cap{^\sz\bfU^+}[i]$, then $T_i:\bfU^+_\cZ[i]\ra{^\sz\bfU^+_\cZ}[i]$ is an isomorphism and $T_i^{-1}$ is its inverse.

\subsection{BGP reflections functors}\label{BGP relection functors}
A vertex $i\in I$ is called a sink (resp. source) of $Q$ if there is no arrow $h\in H$ with $s(h)=i$ (resp. $t(h)=i$).
We define $\sz_i Q$ to be the quiver obtained from $Q$ by reversing the direction of every arrow connected to the vertex $i$.
In \cite{Bernstein_Gel'fand_Ponomarev_Coxeter_functors_and_Gabriel's_theorem}, one may define the BGP reflection functors:
$$\sz^+_i:\mod kQ\ra\mod k\sz_i Q$$
if $i$ is a sink and
$$\sz^-_i:\mod kQ\ra\mod k\sz_i Q$$
if $i$ is a source.
Let $\mod kQ[i]$ be the subcategory of $\mod kQ$ consisting of all representations which do not have $S_i$ as a direct summand and $\cH^*(kQ)[i]$ be the subalgebra of $\cH^*(kQ)$ generated by $u_{[M]}$ with $M\in\mod kQ[i]$.
Therefore we have algebra homomorphisms:
$$\sz_i^+:\cH^*(kQ)[i]\ra\cH^*(k\sz_i Q)[i]$$
if $i$ is a sink and
$$\sz_i^-:\cH^*(kQ)[i]\ra\cH^*(k\sz_i Q)[i]$$
if $i$ is a source, defined by
$$\sz_i^\pm(\lan M\ran)=\lan\sz_i^\pm(M)\ran$$
for any $M\in\mod kQ[i]$. Furthermore, they induce the homomorphisms
$$\sz_i^+:\cC^*(Q)[i]\ra\cC^*(\sz_iQ)[i]$$
if $i$ is a sink and
$$\sz_i^-:\cC^*(Q)[i]\ra\cC^*(\sz_iQ)[i]$$
if $i$ is a source.
Here $\cC^*(Q)[i]=\cC^*(Q)\cap\prod_{q\in\cW}\cH^*(\bbF_q Q)[i]$.

It is known that $\sz_i^+=T_i$ and $\sz_i^-=T_i^{-1}$ under the identification $\Phi:\bfU^+[i]\ra\cC^*(Q)[i]$ when $i$ is a sink and $\Phi:{^\sz\bfU^+[i]}\ra\cC^*(Q)[i]$ when $i$ is a source (see \cite{Deng_Xiao_Ringel-Hall_algebras_and_Lusztig's_symmetries} or \cite{Xiao_Yang_BGP-reflection_functors_and_Lusztig's_symmetries}).

\section{Cyclic quivers and PBW bases}\label{cycliccase}

\subsection{Nilpotent representations of cyclic quivers}
Let $k$ be a field, $\dt$=$\dt(n)$ be the cyclic quiver

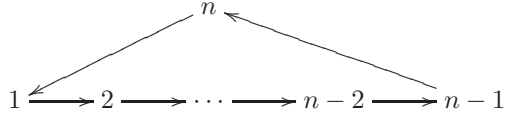
\begin{equation}\notag
\xymatrix{
           &         &    n\ar[lld]       &         &       \\
1\ar[r] & 2\ar[r] & \cdots \ar[r]& n-2\ar[r]&  n-1\ar[llu]\\
}
\end{equation}
with vertex set $I=\bbZ/n\bbZ=\{1,2,\dots,n\}$ and arrow set $H=\{i\ra i+1|1\leq i\leq n\}$. Let $\cK_n=\cK_n(k)$ be the category of nilpotent representations of the path algebra $k\dt$.

A multisegment is a formal finite sum
$$\pi=\sum_{i\in I,l\geq 1}\pi_{i,l}[i;l),$$
where $\pi_{i,l}\in\bbN$ and there are only finitely many nonzero $\pi_{i,l}$. Let $\Pi=\Pi(n)$ denote the set of all multisegments. Each multisegment $\pi=\sum_{i\in I,l\geq 1}\pi_{i,l}[i;l)\in\Pi$ defines a representation in $\cK_n$:
$$M(\pi)=M_{\cK_n}(\pi)=\bigoplus_{i\in I,l\geq 1}S_i[l]^{\oplus \pi_{i,l}},$$
where $S_i[l]$ is the indecomposable $k\dt$-module with top $S_i$ and of length $l$.
This gives a bijection between $\Pi$ and the set of isomorphim classes of representations in $\cK_n$ and the bijection is independent of the field $k$.

\subsection{Aperiodicity}
A multisegment $\pi=\sum_{i\in I,l\geq 1}\pi_{i,l}[i;l)$ is called aperiodic if, for each $l\geq 1$, there is some $i\in I$ such that $\pi_{i,l}=0$, otherwise, $\pi$ is called periodic.
By $\Pi^a$ we denote the set of aperiodic multisegments.
A representation $M$ in $\cK_n$ is called aperiodic (resp. periodic) if $M\cong M(\pi)$ for some $\pi\in\Pi^a$ (resp. $\pi\in\Pi\setminus\Pi^a$).

\subsection{Order of multisegements}
For $\nu\in \bbN I$, let
$$\Pi_\nu=\{\pi\in\Pi | \udim M(\pi)=\nu\}, \Pi^a_\nu=\Pi^a\cap\Pi_\nu.$$
Define a partial order $\leq_G$ on $\Pi_\nu$.
For $\pi,\pi'\in\Pi_\nu$, $\pi'\leq_G\pi$ if and only if $$\dim\,\Hom(M,M(\pi'))\geq\dim\,\Hom(M,M(\pi))$$ for all modules (or equivalently, for all indecomposable modules) $M$ in $\cK_n$.
By \cite{Bongartz_On_degenerations_and_extensions_of_finite-dimensional_modules}, this ordering coincides with the geometric ordering: $\pi'\leq \pi$ if and only if the orbit of $M(\pi')$ in the representation space is contained in the closure of the orbit of $M(\pi)$.

\subsection{Generic extensions}
For any two modules $M,N\in\cK_n$, there is a unique (up to isomorphism) extension $L$ of $M$ by $N$ with minimal $\dim\,\End\, L$. The module $L$ is called the generic extension of $M$ by $N$ and we denote $L=M\diamond N$ or $[L]=[M]\diamond[N]$.

\subsection{Bases of $\cC^*$}\label{PBW basis of cyclic quiver}
A PBW basis of $\cC^*$ is constructed in \cite{Deng_Du_Xiao_Generic_extensions_and_canonical_bases_for_cyclic_quivers}, where $\cC^*$ is the composition algebra of the category $\cK_n$.

\begin{proposition}[\cite{Deng_Du_Xiao_Generic_extensions_and_canonical_bases_for_cyclic_quivers}]
For each $\pi\in\Pi^a$, there is a distinguished monomial $\fkm^{\oz_\pi}$ satisfying
$$S_{i_1}^{\oplus a_1}\diamond S_{i_2}^{\oplus a_2}\diamond\dots\diamond S_{i_s}^{\oplus a_s}=M(\pi)$$
and
$$\fkm^{\oz_\pi}=\lan M(\pi)\ran+\sum_{\pi'\in\Pi_\nu,\pi'<_G \pi}\xz^{\pi'}_{\oz_\pi}(v)\lan M(\pi')\ran,$$
where $\oz_\pi=((i_1,\dots,i_s),(a_1,\dots,a_s))$, $\xz^{\pi'}_{\oz_\pi}\in\cZ$.
\end{proposition}

Since $\Pi^a_\nu$ is a finite set for any $\nu\in\bbN I$, define $E_{\pi}$ for all $\pi\in\Pi^a_\nu$ inductively by the following relations.
If $\pi\in\Pi^{a}_\nu$ is minimal,
$$E_{\pi}={\fkm}^{w_{\pi}}=\langle{M(\pi)}\rangle+\sum_{\pi'\in\Pi_\nu\setminus\Pi_\nu^a,\pi'<_G\pi}\xi_{\oz_\pi}^{\pi'}(v)\langle{M(\pi')}\rangle.$$
If $E_{\pi'}$ have been defined for all $\pi>\pi'\in\Pi_\nu^a$, then
\begin{eqnarray*}
E_{\pi}&=&{\fkm}^{w_{\pi}}-\sum_{\pi'\in\Pi_\nu^a,\pi'<_G\pi}\xi_{\oz_\pi}^{\pi'}(v)E_{\pi'}\\
&=&
\langle{M(\pi)}\rangle+\sum_{\pi'\in\Pi_\nu\setminus\Pi_\nu^a,\pi'<_G\pi}\eta_{\oz_\pi}^{\pi'}(v)\langle{M(\pi')}\rangle,
\end{eqnarray*}
for some $\eta^{\pi'}_{\oz_\pi}\in\cZ$.

\begin{proposition}[\cite{Deng_Du_Xiao_Generic_extensions_and_canonical_bases_for_cyclic_quivers}]
The set $\{E_\pi|\pi\in\Pi^a\}$ is a $\cZ$-basis of $\cC^*_\cZ(\dt)$ and is independent of the choice of distinguished monomials.
\end{proposition}

\section{Representations of affine quivers}

\subsection{Representations of affine quivers}
Let $Q=(I,H,s,t)$ be an acyclic quiver of affine type.
Let $k$ be a finite field and $kQ$ be the path algebra.

Denote by $\ind\cC$ the set of isomorphism classes of all indecomposable objects in the category $\cC$. According to \cite{Ringel_Tame_Algebras_and_Integral_Quadratic_Forms}, the set $\ind\mod kQ$ can be divided into
$$\ind\mod kQ=\ind\cP\sqcup\ind\cR_{\rm{nh}}\sqcup\ind\cR_{\rm{h}}\sqcup\ind\cI,$$
where $\cP=\cP(k)$, $\cR_{\rm{nh}}=\cR_{\rm{nh}}(k)$, $\cR_{\rm{h}}=\cR_{\rm{h}}(k)$, $\cI=\cI(k)$ denotes the full subcategory of $\mod kQ$ consisting of all preprojective, non-homogeneous regular, homogeneous regular, preinjective modules respectively.
Hence we have
$$\cH^*(kQ)=\cH^*_q(\cP)\otimes_{\bbQ(v_q)}\cH^*_q(\cR_{\rm{nh}})\otimes_{\bbQ(v_q)}\cH^*_q(\cR_{\rm{h}})\otimes_{\bbQ(v_q)}\cH^*_q(\cI),$$
where $\cH^*_q(\cP)$, $\cH^*_q(\cR_{\rm{nh}})$, $\cH^*_q(\cR_{\rm{h}})$, $\cH^*_q(\cI)$ are the Ringel-Hall algebras of corresponding subcategories and each of them is closed under multiplication in $\cH^*(kQ)$.

The following lemma gives an easy method to check which subcategory an indecomposable module is in.
\begin{lemma}[\cite{Simson_Skowronski_Elements_of_the_representation_theory_of_associative_algebras}]\label{defect}
Let $M$ be an indecomposable $kQ$-module. Then
\begin{enumerate}
\item[(a)] $\lan \dz,\udim M\ran<0\Leftrightarrow M\in\cP$;
\item[(b)] $\lan \dz,\udim M\ran>0\Leftrightarrow M\in\cI$;
\item[(c)] $\lan \dz,\udim M\ran=0\Leftrightarrow M\in\cR_{\rm nh}$ or $\cR_{\rm h}$;
\item[(d)] $M\in\cR_{\rm h}\Rightarrow \udim M=m\dz$ for some $m\in\bbN$.
\end{enumerate}
\end{lemma}

\subsection{Preprojective and preinjective components}\label{PI components}

A doubly infinite sequence
$$\underline{i}=(\dots,i_{-1},i_0|i_1,i_2,\dots)$$
in $I$ is said to be reduced if for any integers $r\leq t$, $s_{i_r}s_{i_{r+1}}\dots s_{i_t}$ is reduced in the Weyl group.

We say a finite sequence ($i_m,\dots,i_1$) is an admissible sink (resp. source) sequence of $Q$ if $i_m$ is a sink (resp. source) of $Q$, and  $i_t$ is a sink (resp. source) of $\sz_{i_{t+1}}\dots\sz_{i_m}Q$ for any $1\leq t<m$.
We say that $\underline{i}$ is adapted to $Q$ if for any $t\leq0$, ($i_0,i_{-1},\dots,i_t$) is an admissible sink sequence and for any $t>0$, ($i_1,i_2,\dots,i_t$) is an admissible source sequence. We say $\underline{i}$ is admissible if $\underline{i}$ is reduced and adapted to $Q$.

Such an admissible $\underline{i}$ always exists for any acyclic quiver $Q$. For example, if we write $I=\{1,2,\dots,n\}$ such that there are no arrows from $i$ to $j$ for $i>j$, then let $i_t=t$ for $1\leq t\leq n$, and $i_{t+n}=i_t$. One can check that this $\underline{i}$ is admissible.

For an admissible $\underline{i}$, set
$$\bz_t=\left\{
\begin{aligned}
&s_{i_0}s_{i_{-1}}\dots s_{i_{t+1}}(i_t),   & \mathrm{if}\, t\leq 0,\\
&s_{i_1}s_{i_{2}}\dots s_{i_{t-1}}(i_t),    & \mathrm{if}\, t>0.
\end{aligned}
\right.
$$
in $\bbN I$. Here $\bz_0=i_0,\bz_1=i_1$.
We order these $\bz_t,t\in\bbZ$ by
$$(\bz_0<\bz_{-1}<\bz_{-2}<\cdots)<(\cdots<\bz_3<\bz_2<\bz_1).$$
For each $t\leq 0$, there is an indecomposable preprojective module $$M(\bz_t)=\sz_{i_0}^-\sz_{i_{-1}}^-\dots \sz_{i_{t+1}}^-(S_{i_t})$$
 with dimension vector $\bz_t$. For each $t>0$, there is an indecomposable preinjective module $$M(\bz_t)=\sz_{i_1}^+\sz_{i_{2}}^+\dots \sz_{i_{t-1}}^+(S_{i_t})$$
with dimension vector $\bz_t$. All indecomposable preprojectve (resp. preinjective) modules (up to isomorphisms) are obtained in this way (see \cite{Bernstein_Gel'fand_Ponomarev_Coxeter_functors_and_Gabriel's_theorem} and \cite{Ringel_PBW-bases_of_quantum_groups}).
So
$$\ind\cP=\{[M(\bz_t)]|t\leq 0\}, \ind\cI=\{[M(\bz_t)]|t>0\}.$$
In addition, by Section \ref{BGP relection functors},
$$\lan M(\bz_t)^{\oplus m}\ran=T_{i_0}^{-1}T_{i_{-1}}^{-1}\dots T_{i_{t+1}}^{-1}(E_{i_t}^{(m)}),t\leq 0,$$
$$\lan M(\bz_t)^{\oplus m}\ran=T_{i_1}T_{i_{2}}\dots T_{i_{t-1}}(E_{i_t}^{(m)}),t>0,$$
under the identification $\Phi^+$.

Set $\cG_-$  to be the set of all finitely supported functions $\bbZ_{\leq 0}\ra\bbN$ and $\cG_+$ to be the set of all finitely supported functions $\bbZ_{>0}\ra\bbN$.
Given $\bfc_-\in\cG_-$ and $\bfc_+\in\cG_+$, define
$$M(\bfc_-)=\oplus_{t\leq 0}M(\bz_t)^{\oplus \bfc_-(t)}\in\cP,$$
$$M(\bfc_+)=\oplus_{t>0}M(\bz_t)^{\oplus \bfc_+(t)}\in\cI.$$

Define the lexicographic orderings on $\cG_-$ and $\cG_+$.
Given $\bfc_-,\bfd_-\in\cG_-$,
we say $\bfc_->_L\bfd_-$ if there exists $t\leq 0$ such that $\bfc_-(t)>\bfd_-(t)$ and $\bfc_-(l)=\bfd_-(l)$ for all $t< l\leq0$.
Given $\bfc_+,\bfd_+\in\cG_+$,
we say $\bfc_+>_L\bfd_+$ if there exists $t>0$ such that $\bfc_+(t)>\bfd_+(t)$ and $\bfc_-(l)=\bfd_-(l)$ for all $0<l<t$.

By Lemma \ref{M^(m)}, $\cH^*(\cP)$ has a natural basis
$$\{\lan M(\bfc_-)\ran=\lan M(\bz_0)\ran^{(\bfc_-(0))}*\lan M(\bz_{-1})\ran^{(\bfc_-(-1))}*\cdots|\bfc_-\in\cG_-\},$$
and $\cH^*(\cI)$ has a natural basis
$$\{\lan M(\bfc_+)\ran=\cdots*\lan M(\bz_2)\ran^{(\bfc_+(2))}*\lan M(\bz_{1})\ran^{(\bfc_+(1))}|\bfc_+\in\cG_+\}.$$

\begin{lemma}\label{preproj}
For $0\geq t'>t$,
$$\lan M(\bz_t)\ran*\lan M(\bz_{t'})\ran=v^{(\bz_{t},\bz_{t'})}\lan M(\bz_{t'})\ran*\lan M(\bz_t)\ran+\sum_{\bfc_-\in\cG_-}a^{\bfc_-}_{t,t'}\lan M(\bfc_-)\ran$$
in $\cH^*(\cP)$, where the sum is taken over all $\bfc_-$ such that $\bfc_-(j)=0$ if $j\geq t'$ or $j\leq t$, and  $a^{\bfc_-}_{t,t'}\in\bbQ(v)$.
\end{lemma}
\begin{proof}
See \cite{Ringel_The_Hall_algebra_approach_to_quantum_groups} for the finite-type case, and the same proof works here.
\end{proof}

\begin{corollary}\label{preprojdecomp}
For $\bfc_-,\bfb_-\in\cG_-$,
$$\lan M(\bfc_-)\ran*\lan M(\bfb_-)\ran=\sum_{\bfd_-\in\cG_-,\bfd_-\geq_L\bfc_-}a^{\bfd_-}_{\bfc_-,\bfb_-}\lan M(\bfd_-)\ran,$$
where $a^{\bfd_-}_{\bfc_-,\bfb_-}\in\bbQ(v)$.
Moreover, if $\bfb_-\neq 0$, then $\bfd_->_L\bfc_-$.
\end{corollary}
\begin{proof}
The proof is the similar to the proof of (3.34) in \cite{Beck_Nakajima_Crystal_bases_and_two-sided_cells_of_quantum_affine_algebras}.
\end{proof}

Dually, we have same conclusions for the preinjective part:

\begin{lemma}\label{preinj}
For $t'>t\geq 0$,
$$\lan M(\bz_t)\ran*\lan M(\bz_{t'})\ran=v^{(\bz_{t},\bz_{t'})}\lan M(\bz_{t'})\ran*\lan M(\bz_t)\ran+\sum_{\bfc_+\in\cG_+}a^{\bfc_+}_{t,t'}\lan M(\bfc_+)\ran$$
in $\cH^*(\cP)$, where the sum is taken over all $\bfc_+$ such that $\bfc_+(j)=0$ if $j\geq t'$ or $j\leq t$, $a^{\bfc_+}_{t,t'}\in\bbQ(v)$.
\end{lemma}
\begin{corollary}\label{preinjdecomp}
For $\bfc_+,\bfb_+\in\cG_+$,
$$\lan M(\bfc_+)\ran*\lan M(\bfb_+)\ran=\sum_{\bfd_+\in\cG_+,\bfd_+\geq_L\bfb_+}a^{\bfd_+}_{\bfc_+,\bfb_+}\lan M(\bfd_+)\ran,$$
where $a^{\bfd_+}_{\bfc_+,\bfb_+}\in\bbQ(v)$.
Moreover, if $\bfc_+\neq 0$, then $\bfd_+>_L\bfb_+$.
\end{corollary}

\subsection{Non-homogeneous tubes}
For $\cR_{\rm{nh}}=\cR_{\rm{nh}}(k)$, there are several extension-closed subcategories $J_1,J_2,\dots,J_s(s\leq 3)$ of $\mod kQ$ known as non-homogeneous tubes such that $\cR_{\rm{nh}}$ is the direct product of them. Thus
$$\cH^*_q(\cR_{\rm{nh}})=\cH^*_q(J_1)\otimes_{\bbQ(v_q)}\cH^*_q(J_2)\otimes_{\bbQ(v_q)}\cdots\otimes_{\bbQ(v_q)}\cH^*_q(J_s)$$
and $xy=yx$ if $x\in\cH^*_q(J_i),y\in\cH^*_q(J_j),$ and $i\neq j$.

It is well-known that each $J_d,1\leq d \leq s$ is equivalent to the category $\cK_r$ of nilpotent representations of the cyclic quiver with $r$ vertices for some $r=r(d)$, and we denote $r=\rk J_d$.
Note that the equivalence is generally not unique when considering the rotation on $\bbZ/r\bbZ$. We choose an equivalence for each $J_d$
and denote it by
$$\ez_d:\cK_{r(d)}\lra J_d.$$
Thus $\ez_d$ induces an algebra isomorphism, which is denoted again by
\begin{align}
\ez_d:\cH^*_q(\cK_{r(d)})&\lra\cH^*_q(J_d)\notag\\
\lan M\ran&\lmto \lan\ez_d(M)\ran.\notag
\end{align}
Using the notations in Section \ref{cycliccase}, there is a bijection from $\Pi(r(1))\times\Pi(r(2))\times\dots\times\Pi(r(s))$ to the set of isomorphism classes of objects in $\cR_{\rm{nh}}$ by mapping $(\pi_1,\pi_2,\cdots,\pi_s)$ to $[\oplus_{1\leq d\leq s}M_{J_d}(\pi_d)]$, where $M_{J_d}(\pi_d)=\ez_d(M_{\cK_r(d)}(\pi_d))$.

\subsection{Homogeneous tubes}\label{homog}
 For $\cR_{\rm{h}}=\cR_{\rm{h}}(k)$, there are extension closed subcategories known as homogeneous tubes, indexed by a subset $Z_k$ of the projective line $\bbP^1(k)$.

A homogeneous regular module $M\in\cR_{\rm{h}}$ can be decomposed uniquely as
$$M=M_k(\underline{\lz},\underline{z})=M_k(\lz^1,z_1)\oplus M_k(\lz^2,z_2)\oplus\cdots\oplus M_k(\lz^t,z_t),$$
where $\underline{\lz}=(\lz^1,\lz^2,\dots,\lz^t)$ is a $t$-tuple of partitions and $z_1,z_2,\dots,z_t\in Z_k$ are distinct.
The modules of the form $M_k((1),z)$ will be called regular simple modules. We set $\deg z=d$ if $\udim M_k((1),z)=d\dz$. When $k$ is algebraic closed, $\deg z=1$ for all $z\in Z_k$.

Now let $k=\bbF_q$. The subalgebra $\cH^*_q(\cR_{\rm{h}})$ of $\cH^*_q(Q)$ is isomorphic to
$$\cH^*_q(Z_k):=\bigotimes_{z\in Z_k}\cH^*_q(z),$$
where $\cH^*_q(z)$ is the classical Hall algebra of homogeneous tube $z$.
Since there are no non-trivial extensions and homomorphisms between two different homogeneous tubes, the algebra $\cH^*_q(\cR_{\rm{h}})$ is commutative.

\subsection{The existence theorem of Hall polynomials}\label{Section:The existence of Hall poly}

We define the index set $\cG'$ consisting of all triples $\bfc=(\bfc_-,\bfc_0,\bfc_+)$, where $\bfc_-\in\cG_-$, $\bfc_+\in\cG_+$, $\bfc_0=(\pi_1,\pi_2,\dots,\pi_s)\in\Pi(r(1))\times\Pi(r(2))\times\cdots\times\Pi(r(s))$.
Following Section \ref{PI components}, define
$$M(\bfc_-)=\oplus_{t\leq 0}M(\bz_t)^{\oplus \bfc_-(t)};$$
$$M(\bfc_+)=\oplus_{t>0}M(\bz_t)^{\oplus \bfc_+(t)};$$
$$M(\bfc_0)=\oplus_{1\leq j\leq s}M_{J_j}(\pi_j);$$
and
$$M(\bfc)=M(\bfc_-)\oplus M(\bfc_0)\oplus M(\bfc_+).$$
This gives a bijection between $\cG'$ and the set of isomorphism classes of $kQ$-modules without homogeneous regular summands.

Let $\cG''$ be the set of all $(\bfc,\underline{\lz})$ where $\bfc\in\cG'$ and $\underline{\lz}=(\lz^1,\lz^2,\dots,\lz^t)$ is a sequence of partitions. Here we allow $\lz^i=0$ for some $i$.

We say $\underline{z}$ is of type $\underline{d}$ if $\deg z_i=d_i$ for all $i$.
Let
$$M(\bfc,\underline{\lz},\underline{z})=M(\bfc)\oplus M_k(\underline{\lz},\underline{z}).$$

The following result indicates the existence of certain Hall polynomials.
\begin{lemma}[\cite{Deng_Ruan_Hall_polynomials_for_tame_type}]\label{hallpoly}
Fix a sequence $\underline{d}=(d_1,d_2,\cdots,d_t)$.
Given $\bm{\az},\bm{\bz},\bm{\gz}\in\cG''$, there exists a polynomial $\varphi^{\bm{\az}}_{\bm{\bz},\bm{\gz}}\in\bbZ[T]$ such that for each finite field $k=\bbF_q$,
$$\varphi^{\bm{\az}}_{\bm{\bz},\bm{\gz}}(q)=g^{M(\bm{\az},\underline{z})}_{M(\bm{\bz},\underline{z}),M(\bm{\gz},\underline{z})}$$
for all $\underline{z}$ of type $\underline{d}$.
\end{lemma}

In particular, we obtain the following corollary.
\begin{corollary}
For $\bfc_1,\bfc_2,\bfc_3\in\cG'$, the Hall number $g^{M(\bfc_1)}_{M(\bfc_2),M(\bfc_3)}$ over field $\bbF_q$ is a polynomial in $q$.
\end{corollary}

\section{Irreducible characters arising from $\cH^*(\cR_{\rm h})$}\label{Chapter: irreducible characters}

\subsection{Generators in $\cH^*(\cR_{\rm h})$}
Let $Q$ be an acyclic affine quiver, $k$ be a finite field and $\cR_{\rm h}=\cR_{\rm h}(k)$ be the homogeneous regular subcategory of $\mod kQ$.
When necessary, we require that $q=|k|$ is large enough.
For $m\in\bbN$, consider
$$H_m=H_m(k)=\sum_{[M]\atop M\in\cR_{\rm h},\udim M=m\dz}v_q^{-\dim M}u_{[M]}\in\cH^*_q(\cR_{\rm h}).$$
Note that this is a finite sum since $k$ is a finite field.

For a partition $\lz=(\lz_1,\lz_2,\dots,\lz_s)$, define
$$H_\lz=\prod_{1\leq k\leq s}H_{\lz_k},$$

We can write in terms of modules
$$H_\lz=\sum_{[M]\atop M\in\cR_{\rm h},\udim M=m\dz}A^{[M]}_\lz u_{[M]},$$

We want to compute $A^{[M]}_\lz$  for some special homogeneous regular $M$.

\subsection{The character of permutation modules arising from $\cH^*(\cR_{\rm h})$}\label{sec: character of permutation mod}
We introduce some notions about representation theory of the symmetric groups first and we refer to \cite{GTM203_The_Symmetric_Group} for more details.

Let $\mathfrak{S}_m$ be the symmetric group of $m$.
For a partition $\lz=(\lz_1\geq \lz_2\geq\dots\geq \lz_s)$ of $m$, a $\lz$-tabloid is by definition an ordered sequence $(W_1,W_2,\dots,W_s)$ of subsets $W_i\subset\{1,2,\dots,m\}$ such that $\{1,2,\dots,m\}=\sqcup_{1\leq i\leq s} W_i$ and $|W_i|=\lz_i$.
This definition is equivalent to the one in \cite{GTM203_The_Symmetric_Group}.
Then the permutation module $M^\lz$ of $\mathfrak{S}_{|\lz|}$ is the $\mathbb{C}$-linear space spanned by all $\lz$-tabloids.
The natural action of $\mathfrak{S}_{m}$ on $\{1,2,\dots,m\}$ induces the action of $\mathfrak{S}_{m}$ on $M^\lz$.

Let $S^\lambda$ be the Specht module introduced in Section 2.3 in \cite{GTM203_The_Symmetric_Group}. It is known that
$$M^\lz\cong\bigoplus_\mu K_{\lz\mu}S^\mu,$$
where $K_{\lz\mu}$ is the Kostka number.

Let $t_\lambda(\mu)=\chi_{S^\lambda}(g_\mu)$ be the complex character value of $S^\lambda$ at $g_\mu\in\mathfrak{S}_m$
of cycle type $\mu$.
Then $(t_\lambda(\mu))_{\lambda,\mu}$ is the irreducible character table of $\mathfrak{S}_m$.
Let $T_\lambda$ be the character of the permutation module $M^\lambda$.
It follows that $T_\lambda=\sum_{\mu}K_{\lambda\mu}t_\mu$.

Let $Z_K$ be the set of all homogeneous tubes of $\mod KQ$ for any field $K$. Let $k$ be a finite field.
For a partition $\mu=(\mu_1\geq\mu_2\geq\cdots\geq\mu_l)$ and $\underline{z}=(z_1,z_2,\cdots,z_l)$ such that $z_i\in Z_k$ are distinct and $\deg z_i=\mu_i$ for all $i$, we denote
$$M[\mu,\underline{z}]=M_k((1),z_1)\oplus M_k((1),z_2)\oplus\cdots\oplus M_k((1),z_l).$$

\begin{proposition}\label{Character: permutation module}
  For partitions $\lambda,\mu=(\mu_1\geq\mu_2\geq\cdots\geq\mu_l)$ with $|\lambda|=|\mu|=m$ and $\underline{z}=(z_1,z_2,\cdots,z_l)$ such that $z_i\in Z_k$ are distinct and $\deg z_i=\mu_i$ for all $i$, we have $$A_\lambda^{[M[\mu,\underline{z}]]}=v_q^{-m|\dz|}T_\lambda(\mu).$$
\end{proposition}

\begin{proof}
Fix $\underline{z}$.
Let $\cR_{\rm h}(K)$ be the homogeneous regular subcategory of $\mod KQ$ for any field $K$.

Let $M=M[\mu,\underline{z}]\in\cR_{\rm h}(k)$. Consider $M'=M[\mu,\underline{z}]\otimes_k \overline{k}\in\cR_{\rm h}(\overline{k})$, then there exists  distinct $x_1,x_2,\dots,x_m\in Z_{\overline{k}}$ such that
$$M'\cong M_{\overline{k}}((1),x_1)\oplus M_{\overline{k}}((1),x_2)\oplus\dots M_{\overline{k}}((1),x_m).$$

Let $\fkf$ be the Frobenius map $\fkf:\overline{k}\ra\overline{k}$ such that $k$ is the fixed point set of $\fkf$.
Then $\fkf$ induces a Frobenius map  $\fkf_{M'}: M'\ra M'$ such that
$$\fkf_{M'}(m\otimes a)=m\otimes \fkf(a)$$
for all $a\in\overline{k}$ and $m\in M[\mu,\underline{z}]$.
This induces a permutation $g_\fkf$ of $x_1,x_2,\dots,x_m$. Then the cycle type of $g_\fkf$ is $\mu$.

Let $\lz=(\lz_1\geq \lz_2\geq\dots\geq \lz_s)$. Define $F^{M'}_\lz$ to be the set of all submodule sequences of the form $0=N'_0\subset N'_1\subset \dots\subset N'_s=M'$ such that $N_i\in\cR_{\rm h}(\overline{k})$ and $\udim N'_i/N'_{i-1}=\lz_i\dz$ for $1\leq i\leq s$. Similarly, define $F^{M}_\lz$ to be the set of all submodule sequences of the form $0=N_0\subset N_1\subset \dots\subset N_s=M$ such that $N_i\in\cR_{\rm h}(k)$ and $\udim N_i/N_{i-1}=\lz_i\dz$ for $1\leq i\leq s$. Then there is an injection
$$F^M_\lz\lra F^{M'}_\lz$$
by mapping $0=N_0\subset N_1\subset \dots\subset N_s=M$ to $0=N_0\otimes_k \overline{k}\subset N_1\otimes_k \overline{k}\subset \dots\subset N_s\otimes_k \overline{k}=M'$. Note that all $N_i\otimes_k \overline{k}$ are $\fkf_{M'}$-stable, and any $\fkf_{M'}$-stable submodule of $M'$ is of the form $N\otimes_k \overline{k}$ where $N$ is a submodule of $M$, then we have a bijection
$$F^M_\lz\lra\{\fkf_{M'}-\text{stable\, submodule\, sequences\, in\, } F^{M'}_\lz\}.$$

On the other hand, since all $M_{\overline{k}}((1),x_i),1\leq i\leq s$ are simple objects in $\cR_{\rm h}(\overline{k})$, any submodule $N$ of $M'$ in $\cR_{\rm h}(\overline{k})$ is of the form $N=\oplus_{x\in U}M_{\overline{k}}((1),x)$, where $U\subset\{x_1,x_2,\dots,x_m\}$. Then there is a bijection
$$F^{M'}_\lz\lra\{\lz-\text{tabloids}\}$$
by mapping $0=N'_0\subset N'_1\subset \dots\subset N'_s=M'$ to $(W_1,W_2,\dots,W_s)$, where $N'_i=\oplus_{x\in U_i}M_{\overline{k}}((1),x)$ and $W_i=U_i\setminus U_{i-1}$ for all $i$. Here we identify $\{1,2,\dots,m\}$ and $\{x_1,x_2,\dots,x_m\}$.
Note that $g_\fkf$ is induced by $\fkf_{M'}$, we have a bijection
$$\{\fkf_{M'}-\text{stable\, submodule\, sequences\, in\, } F^{M'}_\lz\}\lra\{g_\fkf-\text{stable\, }\lz-\text{tabloids}\}.$$

Therefore, by definition, we have $$v_q^{m|\dz|}A_\lambda^{[M[\mu,\underline{z}]]}=|F^M_\lz|=|\{g_\fkf-\text{stable\, }\lz-\text{tabloids}\}|=\chi_{M^\lz}(g_\fkf)=T_\lz(\mu).$$

\end{proof}

\subsection{The character of Specht modules and $S_\lz$}
Let $$a_\lambda^{[M[\mu,\underline{z}]]}=v_q^{-m|\dz|}t_\lambda(\mu).$$
Since $$T_\lambda=\sum_{\lz'}K_{\lambda\lz'}t_{\lz'},$$
these decomposition formulas induce the following formulas
$$A_\lambda^{[M[\mu,\underline{z}]]}=\sum_{\lz'}K_{\lambda\lz'}a_{\lz'}^{[M[\mu,\underline{z}]]}.$$
Recall that $\cH^*_q(\cR_{\rm h})$ is commutative. For a partition $\lz=(\lz_1,\lz_2,\dots,\lz_s)$, define
$$S_\lz=\det(H_{\lz_k-k+m})_{1\leq k,m\leq s}.$$
It is known that
$$H_\lz=\sum_{\lz'}K_{\lz\lz'}S_{\lz'}.$$
Hence, the coefficient of $u_{[M[\mu,\underline{z}]]}$ is $a_\lambda^{[M[\mu,\underline{z}]]}$, when we write $S_\lz$ in terms of $u_{[M]}$ with $M\in\cR_{\rm h}$ such that $\udim M=m\dz$.
This shows that it is reasonable to index $H_\lz$ and $S_\lambda$ by $T_\lz$ and $t_\lambda$ respectively.

\subsection{}
We write
$$S_\lz=\sum_{[M]\atop M\in\cR_{\rm h},\udim M=m\dz}a^{[M]}_\lz u_{[M]}.$$
For a partition $\mu=(\mu_1\geq\mu_2\geq\cdots\geq\mu_l)$ and $\underline{z}=(z_1,z_2,\cdots,z_l)$ such that $z_i\in Z_k$ are distinct and $\deg z_i=1$ for all $i$,  we denote
$$M(\mu,\underline{z})=M_k((\mu_1),z_1)\oplus M_k((\mu_2),z_2)\oplus\cdots\oplus M_k((\mu_l),z_l).$$

We need the following lemma to prove Proposition \ref{multiofN}.
\begin{lemma}\label{Character:Kostka}
For partitions $\lambda,\mu$ with $|\lambda|=|\mu|$ and $\underline{z}=(z_1,z_2,\cdots,z_l)$ such that $z_i\in Z_k$ are pairwise different and $\deg z_i=1$ for all $i$, $a_\lambda^{[M(\mu,\underline{z})]}=v_q^{-|\lambda||\delta|}K_{\mu\lambda}$.
\end{lemma}

\begin{proof}
Let $Z^1_k$ be the subset of $Z_k$ consisting of all $z$ with $\deg z=1$ and $\bbQ(v_q)[Z^1_k]$ be the polynomial algebra with variables in $Z^1_k$.
Consider the $\bbQ(v_q)$-linear map:
$$\cH^*_q(\cR_{\rm h})\stackrel{\phi}{\lra}\bbQ(v_q)[Z^1_k]$$
$$\sum_{[M],M\in\cR_{\rm h}}C^{[M]} u_{[M]}\mapsto \sum_{\mu,\underline{z}}C^{[M(\mu,\underline{z})]}\prod_{i}z_i^{\mu_i},$$
where the sum $\sum_{\mu,\underline{z}}$ runs over all partitions $\mu$ and $\underline{z}=(z_1,z_2,\cdots,z_l)$ such that $z_i\in Z^1_k$ are distinct.
 Since $$\phi(u_{[M(a,z)]}*u_{[M(a',z)]})=\phi(u_{[M(a+a',z)]})=z^{a+a'}=\phi(u_{[M(a,z)]})\phi(u_{[M(a',z)]})$$ for any positive integers $a,a'$ and $z\in Z^1_k$, the map $\phi$ is a homomorphism of algebras.

It is also easy to see that
$$\phi(H_m)=v_q^{-m|\delta|}h_m,$$
 where $h_m$ is the complete symmetric polynomial in $\bbQ(v_q)[Z^1_k]$.
So
$$\phi(S_\lambda)=v_q^{-m|\delta|}s_\lambda,$$
where $s_\lambda$ is the Schur symmetric polynomial of $\lambda$ in $\bbQ(v_q)[Z^1_k]$.
By \cite{Macdonald_Symmetric_functions_and_Hall_polynomials}, the coefficient of $\prod_{i}z_i^{\mu_i}$ in $s_\lambda$ is $K_{\mu\lambda}$ for a partition $\mu$, which implies that $a_\lambda^{[M(\mu,\underline{z})]}=v_q^{-|\lambda||\delta|}K_{\mu\lambda}$.
\end{proof}

\subsection{An example}
Consider the dimension vector $3\dz$.
Set $Z^d=Z_k^d=\{z\in Z_k|\deg z=d\}$.
Denote
\begin{eqnarray}
u_{(3)}&=&\sum_{z\in Z^1}u_{[M((3),z)]},\notag\\
u_{(2,1)}&=&\sum_{z_1\in Z^1,z_2\in Z^1\setminus \{z_1\}}u_{[M((2),z_1)\oplus M((1),z_2)]},\notag\\
u_{[1,1,1]}=u_{(1,1,1)}&=&\sum_{\{z_1,z_2,z_3\}\subset Z^1}u_{[M((1),z_1)\oplus M((1),z_2)\oplus M((1),z_3)]},\notag\\
u_{[2,1]}&=&\sum_{z_1\in Z^2,z_2\in Z^1}u_{[M((1),z_1)\oplus M((1),z_2)]},\notag\\
u_{[3]}&=&\sum_{z\in Z^3}u_{[M_{(1),z}]},\notag\\
u_{\left({2\atop 1}\right)}&=&\sum_{z\in Z^1}u_{[M((2,1),z)]},\notag\\
u_{(1^3)}&=&\sum_{z\in Z^1}u_{[M((1,1,1),z)]},\notag\\
u_{(1^2,1)}&=&\sum_{z_1\in Z^1,z_2\in Z^1\setminus \{z_1\}}u_{[M((1,1),z_1)\oplus M((1),z_2)]}.\notag
\end{eqnarray}

By calculation, we have following equations:
\begin{eqnarray}
v^{3|\dz|}S_{(3)}
&=&u_{(3)}+u_{(2,1)}+u_{[1,1,1]}+u_{[2,1]}+u_{[3]}+u_{\left({2\atop 1}\right)}+u_{(1^3)}+u_{(1^2,1)},\notag \\
v^{3|\dz|}S_{(2,1)}
&=&u_{(2,1)}+2u_{[1,1,1]}-u_{[3]}+v^2u_{\left({2\atop 1}\right)}+(v^4+v^2)u_{(1^3)}+(v^2+1)u_{(1^2,1)},\notag \\
v^{3|\dz|}S_{(1,1,1)}
&=&u_{[1,1,1]}-u_{[2,1]}+u_{[3]}+v^6u_{(1^3)}+v^2u_{(1^2,1)}.\notag
\end{eqnarray}
Since $u_{[\mu]}$ is the sum of all $u_{[M[\mu,z]]}$ for $\mu\vdash 3$, the coefficient of $u_{[\mu]}$ in $S_\lz$ is $a^{M[\mu,z]}_\lz$. This shows that the coefficient matrix
$$\left(
\begin{array}{ccc}
1 & 1 & 1 \\
2 & 0 & -1 \\
1 & -1 & 1 \\
\end{array}
\right)$$
is exactly the character table of the symmetric group $\mathfrak{S}_3$.

Also, the coefficient of $u_{(\mu)}$ in $S_\lz$ is $a^{M(\mu,z)}_\lz$. This shows that the entries of the coefficient matrix
$$\left(
\begin{array}{ccc}
1 & 1 & 1 \\
0 & 1 & 2 \\
0 & 0 & 1 \\
\end{array}
\right)$$
are exactly the Kostka numbers.

\section{The index set}\label{Chapter: The index set}

\subsection{The index set}
 Let $\cG$ be the set consisting of all pairs $(\bfc,t_\lz)$, where $\bfc\in\cG'$ and $t_\lz$ is the character of the Specht module $S^\lz$ or $t_\lz=0=\lz$. The definition of $\cG'$ is in Section \ref{Section:The existence of Hall poly}.
 Define $\tilde{\cG}=\cG'\times\bbN$.

For $(\bfc,t_\lz)\in\cG$, we denote
$$D(\bfc,t_\lz)=\udim M(\bfc)+|\lz|\dz\in \bbN I.$$

For $\nu\in \bbN I$, let $\cG_{\nu}$ (resp., $\tilde{\cG}_{\nu}$) be the subset of $\cG$ (resp. $\tilde{\cG}$) consisting of all $(\bfc,t_\lz)$ such that $D(\bfc,t_\lz)=\nu$ (resp. all $(\bfc,m)$ such that $\udim M(\bfc)+m\dz=\nu$). Then there is a decomposition $\cG=\sqcup_\nu\cG_{\nu}$ (resp. $\tilde{\cG}=\sqcup_\nu\tilde{\cG}_{\nu}$).

We say that $\bfc$ (or $\bfc_0$) is aperiodic if all $\pi_j$ in $\bfc_0=(\pi_1,\pi_2,\cdots,\pi_s)$ are aperiodic.
Let $\cG^a$ be the set of all $(\bfc,t_\lz)$ with $\bfc$ aperiodic.

\subsection{Order of indices}
We define partial orders on $\cG_{\nu}$ and $\tilde{\cG}_{\nu}$:
\begin{definition}\label{order}
Define a partial order $\preceq$ on $\tilde{\cG}_{\nu}$ by letting $(\bfc',m')\prec(\bfc,m)$ if and only if one of the followings is satisfied:
\begin{enumerate}
\item[(a)] $\bfc'_\pm>_L \bfc_\pm$, that is, by definition, $\bfc'_-\geq_L \bfc_-$ and $\bfc'_+\geq_L \bfc_+$ but not all equalities hold;

\item[(b)] $\bfc'_-= \bfc_-$, $\bfc'_+=\bfc_+$, $m'<m$;

\item[(c)] $\bfc'_-= \bfc_-$, $\bfc'_+=\bfc_+$, $m'=m$, and $\bfc'_0<_G\bfc_0$. Here $\bfc'_0<_G\bfc_0$ means that   $\pi'_d\leq_G\pi_d$ for all $1\leq d\leq s$ but not all equalities hold, where
    $\bfc'_0=(\pi'_1,\pi'_2,\cdots,\pi'_s),\bfc_0=(\pi_1,\pi_2,\cdots,\pi_s)$.
\end{enumerate}
Define a partial order of $\cG_{\nu}$, also denoted by $\preceq$, by letting $(\bfc',t_{\lz'})\prec(\bfc,t_\lz)$ if and only if $(\bfc',|\lz'|)\prec(\bfc,|\lz|)$ in $\tilde{\cG}_{\nu}$, or $(\bfc',|\lz'|)=(\bfc,|\lz|)$ and $\lz'>\lz$ where the order of partitions is the lexicographic order (which is also the order of Specht modules).

\end{definition}

\subsection{Geometric property of the order}\label{Geometric property of the order}
Let $k$ be a fixed algebraic closed field. Let $\nu=\sum_{i\in I}\nu_ii\in\bbN I$ and
$$\bbE_\nu=\oplus_{h\in H}\Hom(k^{\nu_{s(h)}},k^{\nu_{t(h)}}), \GL_\nu=\oplus_{i\in I}\GL(k^{\nu_i}).$$

For any $g=(g_i)_{i\in I}\in\GL_\nu,x=(x_h)_{h\in H}\in\bbE_\nu$, define $g.x=((g.x)_h)_{h\in H}$ by
$$(g.x)_h=g_{t(h)}x_hg^{-1}_{s(h)}$$
for any $h\in H$. This defines a $\GL_\nu$-action on $\bbE_\nu$.

Note that given any $x\in\bbE_\nu$, $M(x)=(k^{\nu_i},x_h)_{i\in I,h\in H}$ is a representation of $Q$. The isomorphism classes of representations of dimension vector $\nu$ are then in one-to-one correspondence with the $\GL_\nu$-orbits in $\bbE_\nu$.

Given $(\bfc,m)\in\tilde{\cG}_{\nu}$, let $X(\bfc,m)=\{x\in\bbE_{\nu} | M(x)\cong M(\bfc)\oplus{Z_1}\oplus{Z_2}\oplus\dots\oplus{Z_m}$
for some regular simple modules $Z_1,Z_2,\dots,Z_m$ in distinct homogenous tubes$\}$.

Let $Y(\bfc,m)=\{x\in\bbE_{\nu}|M(x)\cong M(\bfc)\oplus Z$
for some $Z\in\cR_{\rm{h}}$ with $\udim Z=m\dz \}$. By the theorem of  Krull-Schimidt, $\bbE_\nu=\sqcup_{(\bfc,m)\in\tilde{\cG}_{\nu}} Y(\bfc,m).$

In this paper, we only consider the Zariski topology.

\begin{lemma}\label{closedcondition}
Fix a $kQ$-module $N$ and $t\in\bbN$, the subsets
$$\{x\in \bbE_\nu|\dime\Hom(M(x),N)\geq t\},$$
and
$$\{x\in \bbE_\nu|\dime\Hom(N,M(x))\geq t\}$$
of $\bbE_\nu$ are closed.
\end{lemma}
\begin{proof}
This follows from the upper semicontinuity of $\bbE_\nu\times\bbE_{\nu'}\ra\bbN,(x,y)\mapsto\dime\Hom(M(x),M(y))$. See also Section 2.1 in \cite{Bongartz_On_degenerations_and_extensions_of_finite-dimensional_modules}.
\end{proof}

\begin{proposition}\label{Lemma: geometric property of the order}
In $\bbE_\nu$, we have $$\overline{X(\bfc,m)}\subset\bcup_{(\bfc',m')\in\tilde{\cG}_\nu,(\bfc',m')\preceq(\bfc,m)}Y(\bfc',m').$$
\end{proposition}
\begin{proof}
Let $s$ be the maximal positive integer s.t. $\bfc_+(s)\neq 0$. By Lemma \ref{closedcondition}, the sets
$$W_j:=\{x\in\bbE_\nu|\dime\Hom(M(\bz_j),M(x))\geq\dime\Hom(M(\bz_j),M(\bfc_+))\}$$
 are closed for all $j\geq 1$ and $X(\bfc,m)\subset W_{j}$. 
Let $$W_{+}:=\bcap_{j=1}^s W_j,$$ then by the definition of the lexicographic order on $\cG_+$ and the fact that $$\Hom(M(\bz_j),M(\bz_{j'}))=0,\,\text{if}\, j'>j\geq 1,$$ we have
$$W_{+}=\bcup_{(\bfc',m')\in\tilde{\cG}_\nu,\bfc'_+\geq_L \bfc_+}Y(\bfc',m').$$
Also, $W_{+}$ is closed and $X(\bfc,m)\subset W_{+}$.

Dually, let
$$W_j':=\{x\in\bbE_\nu|\dime\Hom(M(x),M(\bz_j))\geq\dime\Hom(M(\bfc_-),M(\bz_j))\}$$
for all $j\leq 0$ and
$$W_{-}:=\bcap_{j=0}^{s'} W'_j,$$
then
$$W_{-}=\bcup_{(\bfc',m')\in\tilde{\cG}_\nu,\bfc'_-\geq_L \bfc_-}Y(\bfc',m')$$
is closed and $X(\bfc,m)\subset W_{-}$. Here $s'$ is the minimal negative integer s.t. $\bfc_-(s')\neq 0$.
Therefore,
$$W_{\pm}=W_{+}\cap W_{-}=\bcup_{(\bfc',m')\in\tilde{\cG}_\nu,\bfc'_-\geq_L\bfc_-, \bfc'_+\geq_L \bfc_+}Y(\bfc',m')$$ is closed.
It follows that $$\overline{X(\bfc,m)}\subset W_{\pm}=(\bcup_{\bfc'_\pm>_L \bfc_\pm}Y(\bfc',m'))\bcup(\bcup_{\bfc'_\pm=\bfc_\pm}Y(\bfc',m')).$$
Recall that for those $(\bfc',m')\in\tilde{\cG}_{\nu}$ with $\bfc'_\pm>_L \bfc_\pm$, they satisfy $(\bfc',m')\prec(\bfc,m)$ already.

Next we prove that for those $(\bfc',m')\in\tilde{\cG}_{\nu}$ with $\bfc'_\pm=\bfc_\pm$ and $(\bfc',m')\nprec(\bfc,m)$, $\overline{X(\bfc,m)}\cap Y(\bfc',m')=\emptyset$. The condition can be divided into several cases as follows.

(1) Let $(\bfc',m')\in\tilde{\cG}_{\nu}$ with $\bfc'_\pm=\bfc_\pm$.
Denote $\bfc'_0=(\pi'_1,\pi'_2,\cdots,\pi'_s),\bfc_0=(\pi_1,\pi_2,\cdots,\pi_s)$.
We claim that if there is $1\leq d\leq s$ and $j\in I_r=\{1,2,\dots,r\},r=r(d)$ such that
$$\dim_j\, M_{\cK_r}(\pi'_d)<\dim_j\, M_{\cK_r}(\pi_d),$$
then $\overline{X(\bfc,m)}\cap Y(\bfc',m')=\emptyset$.

One can check that for any $M\in\cK_r$, if $l\geq\dime M$, then
$$\dime\Hom_{\cK_r}(M_{\cK_r}([j;l)),M)=\dim_j\, M.$$
Let $l\geq\dime M_{\cK_r}(\pi'_d)+\dime M_{\cK_r}(\pi_d)$, then
$$\dime\Hom_{\cK_r}(M_{\cK_r}([j;l)),M_{\cK_r}(\pi'_d))<\dime\Hom_{\cK_r}(M_{\cK_r}([j;l)),M_{\cK_r}(\pi_d)).$$

Back in $\mod kQ$, this means that we can find $M'=\ez_d(M_{\cK_r}([j;l)))\in J_d$ such that
$$\dime\Hom(M',M(\bfc'_0))<\dime\Hom(M',M(\bfc_0)),$$
since there are no homomorphisms between different tubes.
Now consider the closed subset
$$\{x\in\bbE_\nu|\dime\Hom(M',M(x))\geq\dime\Hom(M',M(\bfc_0)\oplus M(\bfc_+))\}.$$
Then it contains all points in $X(\bfc,m)$ but no points in $Y(\bfc',m')$ since $\bfc_+=\bfc'_+$, which implies that $\overline{X(\bfc,m)}\cap Y(\bfc',m')=\emptyset$.

(2) Let $(\bfc',m')\in\tilde{\cG}_{\nu}$ with $\bfc'_\pm=\bfc_\pm,m'>m$. By comparing dimensions, the condition of the assertion in (1) is satisfied, so $\overline{X(\bfc,m)}\cap Y(\bfc',m')=\emptyset$.

(3) Let $(\bfc',m')\in\tilde{\cG}_{\nu}$ with $\bfc'_\pm=\bfc_\pm, m'=m$ but $\bfc'_0\nleq_G \bfc_0$.
Denote $\bfc'_0=(\pi'_1,\pi'_2,\cdots,\pi'_s),\bfc_0=(\pi_1,\pi_2,\cdots,\pi_s)$.

If the condition of the assertion in (2) is satisfied, we have $\overline{X(\bfc,m)}\cap Y(\bfc',m')=\emptyset$. Otherwise, assume that $\dim_j\, M_{\cK_r}(\pi'_d)\geq\dim_j\, M_{\cK_r}(\pi_d)$ for all $1\leq d\leq s$ and $1\leq j\leq r(d)$, then these must be equations since $\dime M(\bfc'_0)=\dime M(\bfc_0)$.

So we can assume that $\udim M_{\cK_r}(\pi'_d)=\udim M_{\cK_r}(\pi_d)$ for all $d$. Then $\bfc'_0\nleq_G \bfc_0$ implies that there is a $d$ such that $\pi'_d\nleq_G \pi_d$. By the definition of ordering on $\Pi(r)$, there is $M'\in\cK_r$ such that $\dime\Hom(M',M_{\cK_r}(\pi'_d))<\dime\Hom(M',M_{\cK_r}(\pi_d))$. Then the same discussion in (2) shows that $\overline{X(\bfc,m)}\cap Y(\bfc',m')=\emptyset$.
\end{proof}

There is also a geometric ordering on $\tilde{\cG}$ defined as follows.
\begin{definition}
Define the ordering $\leq$ on $\tilde{\cG}_\nu$ by letting $(\bfc',m')\leq (\bfc,m)$ if and only if $X(\bfc',m')\subset \overline{X(\bfc,m)}$.
\end{definition}

The following result shows that $\preceq$ is a refinement of $\leq$.
\begin{corollary}
For $(\bfc',m'),(\bfc,m)\in\tilde{\cG}_\nu$, if $(\bfc',m')\leq (\bfc,m)$, then $(\bfc',m')\preceq (\bfc,m)$.
\end{corollary}
\begin{proof}
If $(\bfc',m')\leq (\bfc,m)$, then
$$X(\bfc',m')\subset \overline{X(\bfc,m)}=\bcup_{(\bfc'',m'')\in\tilde{\cG}_\nu,(\bfc'',m'')\preceq(\bfc,m)}Y(\bfc'',m'').$$
Note that $X(\bfc',m')\cap Y(\bfc'',m'')$ is nonempty if and only if $(\bfc'',m'')=(\bfc',m')$, hence $(\bfc',m')\preceq (\bfc,m)$.
\end{proof}

\section{The extended composition algebra and PBW basis}
\subsection{The extended composition algebra for a fixed field}\label{def of N(c,lz)}
Let $k=\bbF_q$, $v_q=\sqrt{q}$.
Let $\cH^0=\cH^0_q$ be the $\bbQ(v_q)$-subalgebra of $\cH^*(kQ)$ generated by $u_i=u_{[S_i]}, i\in I$ and $u_{[M]}, M\in \cR_{\rm{nh}}$.
For $(\bfc,t_\lz)\in\cG$ with $\bfc=(\bfc_-,\bfc_0,\bfc_+)$ and $T_\lz$ as in Section \ref{sec: character of permutation mod}, define
$$N(\bfc,T_\lz)=\lan M(\bfc_-)\ran*\lan M(\bfc_0)\ran*H_{\lz}*\lan M(\bfc_+)\ran,$$
$$N(\bfc,t_\lz)=\lan M(\bfc_-)\ran*\lan M(\bfc_0)\ran*S_{\lz}*\lan M(\bfc_+)\ran,$$
$$N(\bfc,0)=\lan M(\bfc)\ran.$$

\begin{lemma}\label{Lemma: Kronecker of Pm}
Let $I$ be an injective module with $\lan\dz,\udim I\ran=1$, and $P$ be the preprojective module with $\udim P=\dz-\udim I$. Then
$$\lan I\ran^{(m)}*\lan P\ran^{(m)}=H_m+\sum_{1\leq l\leq m}E_{l\dz}H_{m-l}+\sum_{\bfc_->0,\bfc_+>0}v^{d_m(\bfc,t'_\lz)}N(\bfc,T_\lz),$$
for some $E_{l\dz}\in\cH^*_q(\cR_{\rm nh})$, $d_m(\bfc,T_\lz)\in\bbZ$.
\end{lemma}
\begin{proof}
See Lemma 3.13(3) and Section 7.1 in \cite{Lin_Xiao_Zhang_Representations_of_tame_quivers_and_affine_canonical_bases}.
\end{proof}

\begin{lemma}[\cite{Zhang_Rational_Ringel-Hall_Algebras_Hall_Polynomials_Of_Affine_Type_and_Canonical_Bases}]\label{Lemma: tN is basis}
The set
$\{N(\bfc,t_\lz)|(\bfc,t_\lz)\in\cG\}$ is a $\bbQ(v_q)$-basis of $\cH^0_q$.
\end{lemma}
\begin{proof}
By Lemma \ref{Lemma: Kronecker of Pm} and using induction, $H_m\in\cH^0_q$. Hence, $N(\bfc,t_\lz)\in\cH^0_q$.
It remains to show that $\cH^0_q$ can be linearly spanned by $\{N(\bfc,t_\lz)|(\bfc,t_\lz)\in\cG\}$.

By Proposition 7.2 and Section 9.2 in \cite{Lin_Xiao_Zhang_Representations_of_tame_quivers_and_affine_canonical_bases}, any element in $\cC^*$ can be linearly spanned by $\{N(\bfc,t_\lz)|(\bfc,t_\lz)\in\cG\}$.
So we need only to prove that for $M\in \cR_{\rm{nh}}$, elements of form $u_{[M]}*\lan M(\bfc_-)\ran$ or $H_m*\lan M(\bfc_+)\ran*u_{[M]}$ can be linearly spanned by $\{N(\bfc,t_\lz)|(\bfc,t_\lz)\in\cG\}$.
The former is true since the extensions of $M$ by $M(\bfc_-)$ are in $\cP\oplus\cR_{\rm nh}$ and there are only finitely many extensions up to isomorphisms, and the latter is true for a similar reason.
\end{proof}

Therefore we have a natural decomposition $$\cH^0_q=\cH^*_q(\cP)\otimes_{\bbQ(v_q)}\cH^*_q(\cR_{\rm{nh}})\otimes_{\bbQ(v_q)}\bbQ(v_q)[H_1,H_2,\cdots]\otimes_{\bbQ(v_q)}\cH^*_q(\cI).$$

\subsection{The integral and generic property}
We are going to prove the main proposition of this section.
\begin{proposition}\label{multiofN}
Let $(\bfc^1,t_{\lz^1}),(\bfc^2,t_{\lz^2})\in\cG$. Then there exists a polynomial $\psi^{(\bfc,t_{\lz})}_{(\bfc^1,t_{\lz^1}),(\bfc^2,t_{\lz^2})}\in\cZ$ for each $(\bfc,t_\lz)\in\cG$ with $\bfc_-\geq_L \bfc^1_-,\bfc_+\geq_L \bfc^2_+$, such that
$$N(\bfc^1,t_{\lz^1})*N(\bfc^2,t_{\lz^2})=\sum_{(\bfc,t_\lz)\in\cG\atop\bfc_-\geq_L \bfc^1_-,\bfc_+\geq_L \bfc^2_+}\psi^{(\bfc,t_\lz)}_{(\bfc^1,t_{\lz^1}),(\bfc^2,t_{\lz^2})}(v_q)N(\bfc,t_\lz)$$
in $\cH^0_q$  for each finite field $k=\bbF_q$ with $q$ large enough.
\end{proposition}

To do so, we introduce a weaker result first.

\begin{lemma}\label{multi of modules}
Let $\bfc^1,\bfc^2\in\cG'$.
Then there exists a polynomial $\psi^{(\bfc,t_\lz)}_{\bfc^1,\bfc^2}\in\cZ$ for each $(\bfc,t_\lz)\in\cG$, such that
$$\lan M(\bfc^1)\ran*\lan M(\bfc^2)\ran=\sum_{(\bfc,t_\lz)\in\cG}\psi^{(\bfc,t_\lz)}_{\bfc^1,\bfc^2}(v_q)N(\bfc,t_\lz)$$
in $\cH^0_q$  for each finite field $k=\bbF_q$ with $q$ large enough.
\end{lemma}
\begin{proof}
By {Lemma} \ref{Character:Kostka} and comparing the coefficients of $\lan M(\bfc)\oplus M_k(\mu,\underline{z})\ran$ where $\underline{z}=(z_1,z_2,\cdots,z_t)$ such that $z_i\in Z^1_k$ are distinct ($\underline{z}$ can be chosen when $|Z^1_k|\geq |\lz|$, therefore we require that $q=|k|$ is large enough), we get

$$v_q^{d}g^{M(\bfc)\oplus M_k(\mu,\underline{z})}_{M(\bfc^1),M(\bfc^2)}(v_q)=\sum_\lz \psi^{(\bfc,t_\lz)}_{\bfc^1,\bfc^2}(v_q)v_q^{-|\lz|}K_{\mu\lz},
$$
where $d=-\dime\End_{kQ} (M(\bfc)\oplus M_k(\mu,\underline{z}))+\dime\End_{kQ} M(\bfc^1)+\dime\End_{kQ} M(\bfc^2)+\lan\udim M(\bfc^1),\udim M(\bfc^2)\ran$.

By Lemma \ref{hallpoly}, $g^{M(\bfc)\oplus M_k(\mu,\underline{z})}_{M(\bfc^1),M(\bfc^2)}$ is a polynomial in $\bbZ[v]$ and is independent of the choice of $\underline{z}$.
Also, $\dime\End_{kQ} (M(\bfc)\oplus M_k(\mu,\underline{z}))$ is an integer independent of $k$ and $\underline{z}$.
Therefore $\psi^{(\bfc,t_\lz)}_{\bfc^1,\bfc^2}$ is in $\cZ$ since $(K_{\mu\lz})_{\lz,\mu}$ is invertible in $\bbZ_{|\lz|\times|\lz|}$.
\end{proof}

\subsection{Proof of Proposition \ref{multiofN}}
We always assume that $q$ is large enough.

By Lemma \ref{Lemma: tN is basis}, there exists $\psi^{(\bfc,t_\lz)}_{(\bfc^1,t_{\lz^1}),(\bfc^2,t_{\lz^2})}(v_q)\in\bbQ(v_q)$  for each $(\bfc,t_\lz)\in\cG$, such that
\begin{equation}\label{prfofmulti}
N(\bfc^1,t_{\lz^1})*N(\bfc^2,t_{\lz^2})=\sum_{(\bfc,t_\lz)\in\cG}
\psi^{(\bfc,t_\lz)}_{(\bfc^1,t_{\lz^1}),(\bfc^2,t_{\lz^2})}(v_q)N(\bfc,t_\lz).
\end{equation}

First we show that the sum runs over those $(\bfc,t_\lz)\in\cG$ with $\bfc_-\geq_L \bfc^1_-,\bfc_+\geq_L \bfc^2_+$.

Let $\bfc^1_{\geq 0}=(0,\bfc^1_0,\bfc^1_+)$.
Then $N(\bfc^1,t_{\lz^1})=\lan M(\bfc^1_{-})\ran*N(\bfc^1_{\geq 0},t_{\lz^1})$.
By (\ref{prfofmulti}) and Corollary \ref{preprojdecomp}, we have
\begin{align}\label{Proof: c_->c^1_-}
& N(\bfc^1,t_{\lz^1})*N(\bfc^2,t_{\lz^2})\\
=&\lan M(\bfc^1_{-})\ran*(N(\bfc^1_{\geq 0},t_{\lz^1})*N(\bfc^2,t_{\lz^2}))\notag\\
=&\sum_{(\bfc',t_{\lz'})\in\cG}\psi^{(\bfc',t_{\lz'})}_{(\bfc^1_{\geq 0},t_{\lz^1}),(\bfc^2,t_{\lz^2})}(v_q)\lan M(\bfc^1_{-})\ran*N(\bfc',t_{\lz'})\notag\\
=&\sum_{(\bfc',t_{\lz'})\in\cG}\psi^{(\bfc',t_{\lz'})}_{(\bfc^1_{\geq 0},t_{\lz^1}),(\bfc^2,t_{\lz^2})}(v_q)\lan M(\bfc^1_-)\ran*\lan M(\bfc'_-)\ran*\lan M(\bfc'_0)\ran*S_{\lz'}*\lan M(\bfc'_+)\ran.\notag\\
=&\sum_{(\bfc',t_{\lz'})\in\cG}\psi^{(\bfc',t_{\lz'})}_{(\bfc^1_{\geq 0},t_{\lz^1}),(\bfc^2,t_{\lz^2})}(v_q)(\sum_{\bfd_-\geq_L \bfc^1_-}a^{\bfd_-}_{\bfc^1_-,\bfc'_-}(v_q)\lan M(\bfd_-)\ran)*\lan M(\bfc'_0)\ran*S_{\lz'}*\lan M(\bfc'_+)\ran.\notag
\end{align}
This shows that all $(\bfc,t_\lz)\in\cG$ in the right side of (\ref{prfofmulti}) must satisfy $\bfc_-\geq_L \bfc^1_-$.
Dually, they should also satisfy $\bfc_+\geq_L \bfc^2_+$.

It remains to show that all $\psi^{(\bfc,t_\lz)}_{(\bfc^1,t_{\lz^1}),(\bfc^2,t_{\lz^2})}$ are polynomials in $\cZ$.
We shall prove this by induction on the dimension vector $D(\bfc^1,t_{\lz^1})+D(\bfc^2,t_{\lz^2})$.

If $D(\bfc^1,t_{\lz^1})+D(\bfc^2,t_{\lz^2})<\dz$, then $t_{\lz^1}=t_{\lz^2}=0$ and Lemma \ref{multi of modules} shows that $\psi^{(\bfc,t_\lz)}_{\bfc^1,\bfc^2}\in\cZ$.
Now assume that all $\psi^{(\bfc,t_\lz)}_{(\bfc^1,t_{\lz^1}),(\bfc^2,t_{\lz^2})}$ are in $\cZ$ for $D(\bfc^1,t_{\lz^1})+D(\bfc^2,t_{\lz^2})<\nu$.

Consider the equation (\ref{Proof: c_->c^1_-}). If $\bfc^1_->0$,
by assumption and Lemma \ref{multi of modules}, all $\psi^{(\bfc',t_{\lz'})}_{(\bfc^1_{\geq 0},t_{\lz^1}),(\bfc^2,t_{\lz^2})}$ and $a^{\bfd_-}_{\bfc^1_-,\bfc'_-}$ in (\ref{Proof: c_->c^1_-}) shall belong to $\cZ$. Hence $\psi^{(\bfc,t_\lz)}_{(\bfc^1,t_{\lz^1}),(\bfc^2,t_{\lz^2})}$ are in $\cZ$ for all $(\bfc,t_\lz)$.
Dually if $\bfc^2_+>0$, this is also true.

In other words, we have proved the following assertion: if $(\bfc^1,t_{\lz^1}),(\bfc^2,t_{\lz^2}),\dots,(\bfc^j,t_{\lz^j})\in\cG$ satisfy
\begin{equation}\label{aaaaaa}
D(\bfc^1,\lz^1)+D(\bfc^2,\lz^2)+\cdots+D(\bfc^j,\lz^j)\leq\nu, \textrm{and\,} \bfc^1_-> 0 \textrm{ or } \bfc^j_+> 0,\end{equation}
then
$$
N(\bfc^1,t_{\lz^1})*N(\bfc^2,t_{\lz^2})*\cdots*N(\bfc^j,t_{\lz^j})=\sum_{(\bfc,t_\lz)\in\cG}\psi^{(\bfc,t_\lz)}_{(\bfc^1,t_{\lz^1}),(\bfc^2,t_{\lz^2}),\cdots,(\bfc^j,t_{\lz^j})}(v_q)N(\bfc,t_\lz),
$$
with $\psi^{(\bfc,t_\lz)}_{(\bfc^1,t_{\lz^1}),(\bfc^2,t_{\lz^2}),\cdots,(\bfc^j,t_{\lz^j})}\in\cZ$.

Now we assume that $D(\bfc^1,t_{\lz^1})+D(\bfc^2,t_{\lz^2})=\nu,\bfc^1_-=0,\bfc^2_+=0$. Then
$$N(\bfc^1,t_{\lz^1})*N(\bfc^2,t_{\lz^2})=S_{\lz^1}*\lan M(\bfc^1_0)\ran*\lan M(\bfc^1_+)\ran*\lan M(\bfc^2_-)\ran*\lan M(\bfc^2_0)\ran*S_{\lz^2}.$$
By Lemma \ref{multi of modules},
$$\lan M(\bfc^1_+)\ran*\lan M(\bfc^2_-)\ran=\sum_{(\bfc^3,t_{\lz^3})\in\cG}\psi^{(\bfc^3,t_{\lz^3})}_{\bfc^1_+,\bfc^2_-}(v_q)N(\bfc^3,t_{\lz^3})$$
with $\psi^{(\bfc^3,t_{\lz^3})}_{\bfc^1_+,\bfc^2_-}\in\cZ$.

For those $(\bfc^3,t_{\lz^3})$ with $\bfc^3_\pm=0$, we have
\begin{align}
&S_{\lz^1}*\lan M(\bfc^1_0)\ran*N(\bfc^3,t_{\lz^3})*\lan M(\bfc^2_0)\ran*S_{\lz^2}\notag\\
&=\lan M(\bfc^1_0)\ran*\lan M(\bfc^3_0)*\lan M(\bfc^2_0)\ran*S_{\lz^1}*S_{\lz^2}*S_{\lz^3}\notag\\
&=\sum_{[M],M\in\cR_{\rm{nh}},|\lz|=|\lz_1|+|\lz_2|+|\lz_3|}\psi^{[M]}_{\bfc^1_0,\bfc^3_0,\bfc^2_0}(v_q)\psi^{\lz}_{\lz^1,\lz^2,\lz^3}\lan M\ran*S_{\lz},\notag
\end{align}
with $\psi^{[M]}_{\bfc^1_0,\bfc^3_0,\bfc^2_0}\in\cZ$, $\psi^{\lz}_{\lz^1,\lz^2,\lz^3}\in\bbZ$.

For those $(\bfc^3,\lz^3)$ with $\bfc^3_\pm\neq 0$, we can compute case by case.

Assume that $\bfc^3_-> 0$.
Then
$$\lan M(\bfc^1_0)\ran*\lan M(\bfc^3_-)\ran=\sum_{(\bfc^4,t_{\lz^4}),\bfc^4_-> 0}\psi^{(\bfc^4,t_{\lz^4})}_{\bfc^1_0,\bfc^3_-}(v_q)N(\bfc^4,t_{\lz^4}),$$
with $\psi^{(\bfc^4,t_{\lz^4})}_{\bfc^1_0,\bfc^3_-}\in\cZ.$
Since $\lan\dz,D(\bfc^4,t_{\lz^4})\ran=\lan\dz,\udim M(\bfc^1_0)+\udim M(\bfc^3_-)\ran<0$, we have $\bfc^4_-> 0$.

By Lemma \ref{Lemma: Kronecker of Pm},
\begin{align*}
H_m=&\lan I\ran^{(m)}*\lan P\ran^{(m)}+\sum_{\bfc_\pm=0,\bfc_0>0,l<m}\psi^{(\bfc_0,l)}_{m}(v_q)\lan M(\bfc_0)\ran*H_l\\
&+\sum_{\bfc_->0,\bfc_+>0}\psi^{(\bfc,t_\lz)}_{m}(v_q)N(\bfc,t_\lz),
\end{align*}
where $\psi^{(\bfc,t_\lz)}_{m}\in\cZ$, and $\psi^{(\bfc_0,l)}_{m}\in\cZ$ by \cite{Zhang_Rational_Ringel-Hall_Algebras_Hall_Polynomials_Of_Affine_Type_and_Canonical_Bases}.
Then by the induction on $m$, we can prove
that for any $\bfc^4_->0$ and $m\geq 0$,
$$H_m*\lan M(\bfc^4_-)\ran=\sum_{(\bfc^5,t_{\lz^5}),\bfc^5_->0}\psi^{(\bfc^5,t_{\lz^5})}_{H_m,\bfc^4_-}(v_q)N(\bfc^5,t_{\lz^5}),$$
with $\psi^{(\bfc^5,t_{\lz^5})}_{H_m,\bfc^4_-}\in\cZ.$

Similarly, we can prove the desired result for other $(\bfc^3,\lz^3)$ with $\bfc^3_\pm\neq 0$.

In this way, we can always write $N(\bfc^1,t_{\lz^1})*N(\bfc^2,t_{\lz^2})$ as a $\cZ$-linear combination of products of $N(\bfc,t_{\lz})$'s satisfying the condition (\ref{aaaaaa}), thus the proposition is proved.

 \qed

\subsection{The generic extended composition algebra}
Let $v$ be an indeterminant and $\cZ=\bbZ[v,v^{-1}]$.
The generic extended composition algebra $\cH^0_\cZ$ is defined as a $\cZ$-algebra with a formal basis $\{N(\bfc,t_\lz)|(\bfc,t_\lz)\in\cG\}$ subject to the relations
$$N(\bfc^1,t_{\lz^1})*N(\bfc^2,t_{\lz^2})=\sum_{(\bfc,t_\lz)\in\cG,\bfc_-\geq_L \bfc^1_-,\bfc_+\geq_L \bfc^2_+}\psi^{(\bfc,t_\lz)}_{(\bfc^1,t_{\lz^1}),(\bfc^2,t_{\lz^2})}(v)N(\bfc,t_\lz)$$
for all $(\bfc^1,t_{\lz^1}),(\bfc^2,t_{\lz^2})\in\cG$,
where $\psi^{(\bfc,t_\lz)}_{(\bfc^1,t_{\lz^1}),(\bfc^2,t_{\lz^2})}\in\cZ$ is given by Proposition \ref{multiofN}.
Let $\cH^0=\cH^0_\cZ\otimes_\cZ\bbQ(v)$.

When $q$ is large enough, the specialization $v\mapsto v_q$ on the relations gives the formulas of multiplication in $\cH^0_q$, therefore the multiplication $*$ on $\cH^0_\cZ$ is associative.

Since each simple module $S_i$ is either in $\cP$, $\cI$ or $\cR_{\rm nh}$, there exist an index $(\bfc_i,0)\in\cG$ such that $M(\bfc_i)=S_i$.
Then the generic composition algebra $\cC^*$ is naturally a subalgebra of $\cH^0$ by the inclusion
\begin{align}
\cC^*&\lra\cH^0\notag\\
u_i&\longmapsto N(\bfc_i,0).\notag
\end{align}

With Proposition \ref{multiofN}, we also have a natural decomposition $$\cH^0_\cZ=\cH^*_\cZ(\cP)\otimes_{\cZ}\cH^*_\cZ(\cR_{\rm{nh}})\otimes_{\cZ}\cZ[H_1,H_2,\cdots]\otimes_{\cZ}\cH^*_\cZ(\cI),$$
where $\cH^*_\cZ(\cP)$, $\cH^*_\cZ(\cR_{\rm{nh}})$, $\cH^*_\cZ(\cI)$ are the subalgebras naturally defined by corresponding components.

\subsection{Almost orthogonality}\label{Section: Almost orthogonality of N}

\begin{lemma}[\cite{Lin_Xiao_Zhang_Representations_of_tame_quivers_and_affine_canonical_bases}]
For $m\geq 0$, we have
$$(H_m,H_m)\in 1+v^{-1}\bbQ[[v^{-1}]]\cap\bbQ(v).$$
\end{lemma}
\begin{proof}
See Lemma 9.2 in \cite{Lin_Xiao_Zhang_Representations_of_tame_quivers_and_affine_canonical_bases}.
\end{proof}
Let $r$ be the coproduct defined in Section \ref{Section: Bilinear form and coprouct}.
\begin{lemma}
For $m\geq 0$,
\begin{eqnarray}
r(H_m)&=&\sum_{0\leq s\leq m}H_{s}\otimes H_{m-s} \nonumber\\
  &&+\text{ terms in } \cH^*(\cR_{\rm h})* \cH^*(\cI)\otimes\cH^*(\cP)*\cH^*(\cR_{\rm h}).\nonumber
\end{eqnarray}
\end{lemma}
\begin{proof}
By Section \ref{Section: Bilinear form and coprouct}, we have
\begin{eqnarray}
r(H_m)&=&v^{-m|\dz|}r(\sum_{[L],L\in\cR_{\rm h},\udim L=m\dz}u_{[L]}) \nonumber\\
&=&v^{-m|\dz|}\sum_{[L],L\in\cR_{\rm h},\udim L=m\dz}\sum_{[M],[N]\atop M,N\in\cR_{\rm h}}g^L_{MN}\frac{a_M a_N}{a_L}u_{[M]}\otimes u_{[N]} \nonumber\\
  &&+\text{ terms in } \cH^*(\cR_{\rm h})* \cH^*(\cI)\otimes\cH^*(\cP)*\cH^*(\cR_{\rm h}).\nonumber
\end{eqnarray}

If $Q$ is the Kronecker quiver with $I=\{0,1\}$,
the element $\tilde{P}_{1,m}$ in \cite{Beck_Nakajima_Crystal_bases_and_two-sided_cells_of_quantum_affine_algebras} coincides with $H_m$ by the construction in Section 3 in \cite{Lin_Xiao_Zhang_Representations_of_tame_quivers_and_affine_canonical_bases}.
So the lemma is true for the Kronecker quiver by Proposition 3.22 in \cite{Beck_Nakajima_Crystal_bases_and_two-sided_cells_of_quantum_affine_algebras}.

This means that for the Kronecker quiver,
\begin{equation}\label{Equation: eqn in coproduct}
\sum_{[L],L\in\cR_{\rm h},\udim L=m\dz}g^L_{MN}\frac{a_M a_N}{a_L}=1
\end{equation}
holds for any $[M],[N]$ with $M,N\in\cR_{\rm h},\udim M+\udim N=m\dz$.

Note that in the case of Kronecker quiver, (\ref{Equation: eqn in coproduct}) holds in any subcategory $\cC$ of $\cR_{\rm h}$ consisting of finitely many homogeneous tubes for any $M$ and $N$ in $\cC$ and for any finite field $k$.
So (\ref{Equation: eqn in coproduct}) holds in the subcategory $\cC$ of $\cR_{\rm h}$ consisting of finite homogeneous tubes for any $M$ and $N$ in $\cC$ and for any finite field $k$ in any affine quiver.
Therefore, (\ref{Equation: eqn in coproduct}) holds for any $M$ and $N$ in $\cR_{\rm h}$ in any affine quiver.
\end{proof}

\begin{lemma}\label{Lemma: almost orthog of S}
For partitions $\lz,\lz'$, we have
$$(S_\lz,S_{\lz'})\in\dz_{\lz\lz'}+v^{-1}\bbQ[[v^{-1}]]\cap\bbQ(v).$$
\end{lemma}
\begin{proof}
See [\cite{Macdonald_Symmetric_functions_and_Hall_polynomials}, Page 91, Exercise 25] and [\cite{Macdonald_Symmetric_functions_and_Hall_polynomials}, Chapter 1, Equation 4.8].
\end{proof}
\begin{proposition}\label{Proposition: almost orthog}
For $(\bfc,t_\lz),(\bfc',t_{\lz'})\in\cG$, we have
$$(N(\bfc,t_\lz),N(\bfc',t_{\lz'}))\in\dz_{\bfc\bfc'}\dz_{\lz\lz'}+v^{-1}\bbQ[[v^{-1}]]\cap\bbQ(v).$$
\end{proposition}
\begin{proof}
By Proposition 9.1 in \cite{Lin_Xiao_Zhang_Representations_of_tame_quivers_and_affine_canonical_bases}, we have
$$(N(\bfc,t_\lz),N(\bfc',t_{\lz'}))=\dz_{\bfc,\bfc'}\frac{v^{2(\dim\End(M(\bfc_-))+\dim\End(M(\bfc_0))+\dim\End(M(\bfc_+)))}}{a_{M(\bfc_-)}a_{M(\bfc_0)}a_{M(\bfc_+)}}(S_\lz,S_{\lz'}).$$
Then the result is true by Lemma \ref{Lemma: almost orthog of S}.
\end{proof}

The property of $\{N(\bfc,t_\lz)|(\bfc,t_\lz)\in\cG\}$ in Proposition \ref{Proposition: almost orthog} is called almost orthogonality.

\section{Monomial bases, the PBW basis and a bar-invariant basis of $\cC^*$}\label{chapter basis}
In this section, we will always choose $k=\bbF_q$ such that $q$ is large enough.  
We always identify $\bff$, $\bfU^+$ and $\cC^*$.
Recall that $\cS$ is the set defined in Section \ref{Section: def of monomial}.

\subsection{Monomials for a dimension vector}
We label $I=\{1,2,\dots,n\}$ such that there are no arrows from $i$ to $j$ for $i>j$. 
Given $\nu=\sum_i \nu_ii\in\bbN I$, we define a monomial 
$$\fkm^{\nu}=u_1^{(\nu_1)}*u_2^{(\nu_2)}*\dots* u_n^{(\nu_n)}$$
in $\cC^*$.

\subsection{Monomials for preprojective component}
Consider the preprojective component first.
For $t\leq 0$, let $\bfe_t$ be the function $\bbZ^{\leq 0}\ra\bbN$ which takes value $1$ on $t$ and $0$ on the others.
Then for $m\in\bbN$, $N((m\bfe_j,0,0),0)=\lan M(m\bfe_j)\ran=\lan M(\bz_j)^{\oplus m}\ran$.

\begin{lemma}\label{monomial of indecomp}
If $m\geq 0$, $j\leq 0$, then
$$\fkm^{m\bz_j}=\lan M(\bz_j)^{\oplus m}\ran+\sum_{\bfc_->_L m\bfe_j}\phi_{m\bfe_j}^{(\bfc,t_\lz)}(v) N(\bfc,t_\lz)$$
for some $\phi_{m\bfe_j}^{(\bfc,t_\lz)}\in \cZ$.
\end{lemma}

\begin{proof}

By Lemma 6.6 in \cite{Lin_Xiao_Zhang_Representations_of_tame_quivers_and_affine_canonical_bases} and Proposition \ref{multiofN}, it suffices to prove that  all $(\bfc,t_\lz)\in\cG_{m\bz_j}$ satisfy $\bfc_-\geq_L m\bfe_j$.
Suppose there is a $(\bfc,t_\lz)\in\cG_{m\bz_j}$ with $\bfc_-<_L m\bfe_j$, i.e., $\bfc_-(l)=0,j<l\leq 0$ and $\bfc_-(j)<m$.
Then $$(m-\bfc_-(j))\bz_j=\sum_{l<j}\bfc_-(l)\bz_l+\udim M(\bfc_0)+|\lz|\dz+\sum_{l>0}\bfc_+(l)\bz_l.$$
Note that $\bz_l=s_{i_0}s_{i_{-1}}\dots s_{i_{l+1}}(i_l)$ for $l\leq 0$, where $\underline{i}=(\dots,i_{-1},i_0|i_1,i_2,\dots)$ is the admissible sequence in Section \ref{PI components}.
By applying $s_{i_j}s_{i_{j+1}}\dots s_{i_{-1}}s_{i_0}$ on both sides, the left side will be negative while the right side is still positive, which leads to a contradiction.
\end{proof}

Now we turn to an arbitrary $\bfd:\bbZ^{\leq 0}\ra\bbN$ in $\cG_-$.
It can be written as $\bfd=\sum_{t\leq 0}\bfd(t)\bfe_t$.
Define $\oz(\bfd)\in\cS$ and the monomial $\fkm^{\oz(\bfd)}$ by letting
$$\fkm^{\oz(\bfd)}=\fkm^{\bfd(0)\bz_0}*\fkm^{\bfd(-1)\bz_{-1}}*\fkm^{\bfd(-2)\bz_{-2}}*\cdots.$$
Note that $\fkm^{\oz(\bfd)}$ is a finite product since $\bfd$ is finitely supported.

\begin{lemma}\label{momomial for proj}
For $\bfd\in\cG_-$, we have
$$\fkm^{\oz(\bfd)}=\lan M(\bfd)\ran+\sum_{\bfc_->_L\bfd}\phi_\bfd^{(\bfc,t_\lz)}(v) N(\bfc,t_\lz)$$
with $\phi_\bfd^{(\bfc,t_\lz)}\in\cZ$.
\end{lemma}

\begin{proof}
Let $j=j_\bfd\leq 0$ be the minimal integer with $\bfd(j)\neq 0$. We shall prove this reult by induction on $j$.

If $j=0$, then $\fkm^{\oz(\bfd)}=u_{i_0}^{(\bfd(0))}=\lan M(\bfd)\ran$ since $\bz_0=i_0$.

Assume that the lemma holds for all $\bfd$ with $j_\bfd\geq l$. Set $\bfd_{\geq l}=\sum_{t\geq l}\bfd(t)\bfe_t$ and $\bfd_{<l}=\sum_{t< l}\bfd(t)\bfe_t$ for any $\bfd\in\cG_-$. For a $\bfd$ with $j_\bfd=l-1$,
by Lemma \ref{monomial of indecomp}, we have
\begin{align}
\fkm^{\oz(\bfd)}=&\fkm^{\oz(\bfd_{\geq l})}*\fkm^{\oz(\bfd(l-1)\bfe_{l-1})}\notag\\
=&\left(\lan M(\bfd_{\geq l})\ran+\sum_{\bfc_->_L\bfd_{\geq l}}\phi_{\bfd_{\geq l}}^{(\bfc,t_\lz)}(v) N(\bfc,t_\lz)\right)\notag\\
*&\left(\lan M(\bz_{l-1})^{\oplus \bfd(l-1)}\ran+\sum_{\bfc'_->_L \bfd(l-1)\bfe_{l-1}}\phi_{\bfd(l-1)\bfe_{l-1}}^{(\bfc',t_{\lz'})}(v) N(\bfc',t_{\lz'})\right).\notag
\end{align}

By Proposition \ref{multiofN}, $N(\bfc,t_\lz)*\lan M(\bz_{l-1})^{\oplus \bfd(l-1)}\ran$ and $N(\bfc,t_\lz)*N(\bfc',t_{\lz'})$ are $\cZ$-linear combinations of some $N(\bfc'',t_{\lz''})$, which satisfy $\bfc''_-\geq_L\bfc_->_L\bfd_{\geq l}$. Note that $\bfc_->\bfd_{\geq l}$ and $\udim M(\bfc_-)\leq\udim M(\bfd_{\geq l})$ imply that there is a $k$ with $l\leq k\leq 0$ such that $\bfc_-(k)>\bfd(k)$, so $\bfc''_->_L\bfd$.

On the other hand,
$$\lan M(\bfd_{\geq l})\ran*N(\bfc',t_{\lz'})=\lan M(\bfd_{\geq l})\ran*\lan M(\bfc'_{-,\geq l})\ran*\lan M(\bfc'_{-,<l})\ran*N((0,\bfc'_0,\bfc'_+),t_{\lz'}).$$
Here, $\bfc'_->_L \bfd(l-1)\bfe_{l-1}$ and $\udim M(\bfc'_-)\leq\udim M(\bfd(l-1)\bfe_{l-1})$ imply that $\bfc'_{-,\geq l}>0$.
Therefore $\lan M(\bfd_{\geq l})\ran*\lan M(\bfc'_{-,\geq l})\ran$ is a $\cZ$-linear combination of $\lan M(\bfc''_{-,\geq l})\ran$ with $\bfc''_{-,\geq l}>_L \bfd_{\geq l}$ by Corollary \ref{preprojdecomp}, and it follows that $\bfc''_->_L\bfd$ if we set $\bfc''_{-,<l}=\bfc'_{-,<l}$.

Since $\lan M(\bfd_{\geq l})\ran*\lan M(\bz_{l-1})^{\oplus \bfd(l-1)}\ran=\lan M(\bfd)\ran$, we get
$$\fkm^{\oz(\bfd)}=\lan M(\bfd)\ran+\sum_{\bfc''_->_L\bfd}\phi_\bfd^{(\bfc'',t_{\lz''})}(v) N(\bfc'',t_{\lz''})$$
and Proposition \ref{multiofN} implies that all $\phi_\bfd^{(\bfc'',t_{\lz''})}\in\cZ$.

\end{proof}

\subsection{Monomials for preinjective component}
Dually, for $\bfd:\bbZ^{>0}\ra\bbN$ in $\cG_+$, 
define $\oz(\bfd)\in\cS$ and the monomial $\fkm^{\oz(\bfd)}$ by
$$\fkm^{\oz(\bfd)}=\dots*\fkm^{\bfd(3)\bz_3}*\fkm^{\bfd(2)\bz_{2}}*\fkm^{\bfd(1)\bz_{1}}.$$
We shall have a dual lemma to Lemma \ref{momomial for proj}:

\begin{lemma}\label{monomial for inj}
For $\bfd\in\cG_+$, we have
$$\fkm^{\oz(\bfd)}=\lan M(\bfd)\ran+\sum_{\bfc_+>_L\bfd}\phi_\bfd^{(\bfc,t_\lz)}(v) N(\bfc,t_\lz)$$
with $\phi_\bfd^{(\bfc,t_\lz)}\in\cZ$.
\end{lemma}
The proof is similar.

\subsection{Monomials for non-homogeneous tubes}\label{Section: Monimials for nh}

Consider a fixed non-homogeneous tube $J_d$ first.
 We simply denote $r=r(d)$ and $\Pi=\Pi(d)$.
Recall that $\cK_r=\cK_r(k)$ is the category of nilpotent representations of the cyclic quiver $\Delta=\Delta(r)$ with $r$ vertices, and $\ez_d:\cK_r\ra J_d$ is the category equivalence.
Let $L_j=\ez_d(S_j)$ for $1\leq j\leq r$. Then $L_j$ is exceptional (see Lemma \ref{M^(m)}).

By Section \ref{cycliccase}, for $\pi\in\Pi^a$ aperiodic, there exists
$$\oz_\pi=((j_1,j_2,\dots,j_s),(a_1,a_2,\dots,a_s)),1\leq j_1,j_2,\dots,j_s\leq r$$ such that
$$\lan S_{j_1}\ran^{(a_1)}*\lan S_{j_2}\ran^{(a_2)}*\dots*\lan S_{j_s}\ran^{(a_s)}=\lan M_{\cK_r}(\pi)\ran+\sum_{\pi'\in\Pi,\pi'<_G \pi}\xz^{\pi'}_{\oz_\pi}(v)\lan M_{\cK_r}(\pi')\ran$$
in $\cH^*(\cK_r)$.
Applying $\ez_d$ to the both sides, we get
$$\lan L_{j_1}\ran^{(a_1)}*\lan L_{j_2}\ran^{(a_2)}*\dots*\lan L_{j_s}\ran^{(a_s)}=\lan M_{J_d}(\pi)\ran+\sum_{\pi'\in\Pi,\pi'<_G \pi}\xz^{\pi'}_{\oz_\pi}(v)\lan M_{J_d}(\pi')\ran$$
in $\cH^*(J_d)$.
However, the left side is no more a monomial in $\cC^*(Q)$, so we need to find a monomial to replace it.

\begin{lemma}\label{monomial in tube of indecomp}
For $1\leq j\leq r$ and $a>0$, we have
$$\fkm^{a\udim L_j}=\lan L_j\ran^{(a)}+\sum_{\bfc_->0,\bfc_+>0}\phi^{(\bfc,t_\lz)}_{L_j,a}(v)N(\bfc,t_\lz)$$
with $\phi^{(\bfc,t_\lz)}_{L_j,a}\in\cZ$.
\end{lemma}
\begin{proof}
By Lemma \ref{defect}, $\lan \dz,D(\bfc,t_\lz)\ran=a\lan \dz,\udim L_j\ran=0$, so either $\bfc_\pm=0$ or $\bfc_->0,\bfc_+>0$.
 If $\bfc_\pm=0$, then $\lan L_j\ran^{(a)}=\lan L_j^{\oplus a}\ran$ is the only $N(\bfc,t_\lz)$ with $D(\bfc,t_\lz)=a\udim L_j$.
\end{proof}

Now define a monomial $\fkm^{\oz(\pi)}$ in $\cC^*$:
$$\fkm^{\oz(\pi)}=\fkm^{a_1\udim L_{i_1}}*\fkm^{a_2\udim L_{i_2}}*\dots*\fkm^{a_s\udim L_{i_s}}.$$

\begin{lemma}\label{monomial in tube}
For $\pi\in\Pi^a$, we have
$$\fkm^{\oz(\pi)}=\lan M_{J_d}(\pi)\ran+\sum_{\pi'\in\Pi,\pi'<_G \pi}\xz^{\pi'}_{\oz_\pi}(v)\lan M_{J_d}(\pi')\ran+\sum_{\bfc_->0,\bfc_+>0}\phi^{(\bfc,t_\lz)}_{\oz(\pi)}(v)N(\bfc,t_\lz)$$
with $\xz^{\pi'}_{\oz_\pi},\phi^{(\bfc,t_\lz)}_{\oz(\pi)}\in\cZ$.
\end{lemma}
\begin{proof}
By Lemma \ref{monomial in tube of indecomp}, it suffices to show that if $\bfc_->0,\bfc_+>0,\lan\dz,D(\bfc,t_\lz)\ran=0,\lan\dz,D(\bfc',t_{\lz'})\ran=0$, then
$$N(\bfc,t_\lz)*N(\bfc',t_{\lz'})=\sum_{\bfc''_->0,\bfc''_+>0}\psi^{(\bfc'',t_{\lz''})}_{(\bfc,t_\lz),(\bfc',t_{\lz'})}(v)N(\bfc'',t_{\lz''}),$$
$$N(\bfc',t_{\lz'})*N(\bfc,t_\lz)=\sum_{\bfc''_->0,\bfc''_+>0}\psi^{(\bfc'',t_{\lz''})}_{(\bfc',t_{\lz'}),(\bfc,t_\lz)}(v)N(\bfc'',t_{\lz''}).$$
Since $\lan\dz,D(\bfc'',\lz'')\ran=0$, $\bfc''_->0$ implies $\bfc''_+>0$ and vise versa. Proposition \ref{multiofN} shows that $\bfc''_-\geq_L \bfc_->0$ in the first equation and $\bfc''_+\geq_L \bfc_+>0$ in the second equation. Hence, we have $\bfc''_->0,\bfc''_+>0$ in both equations.

\end{proof}

Now we turn to an aperiodic $\bfc_0=(\pi_1,\pi_2,\cdots,\pi_s)\in\Pi^a(1)\times\Pi^a(2)\times\cdots\times\Pi^a(s)$.
Define
$$\fkm^{\oz(\bfc_0)}=\fkm^{\oz(\pi_1)}*\fkm^{\oz(\pi_2)}*\cdots*\fkm^{\oz(\pi_s)}.$$
With Lemma \ref{monomial in tube}, Definition \ref{order} and the assertion in the proof of Lemma \ref{monomial in tube}, we can easily deduce a lemma for $\cR_{\rm nh}$.
\begin{lemma}\label{monomial for nh}
For $\bfc_0\in\cG_0$ aperiodic, we have
$$\fkm^{\oz(\bfc_0)}=\lan M(\bfc_0)\ran+\sum_{\bfc'_0\in\cG_0\atop\bfc'_0<_G\bfc_0}\phi^{\bfc'_0}_{\bfc_0}(v)\lan M(\bfc'_0)\ran+\sum_{\bfc''_->0,\bfc''_+>0}\phi^{(\bfc'',t_\lz)}_{\bfc_0}(v)N(\bfc'',t_\lz),$$
with $\phi^{\bfc'_0}_{\bfc_0},\phi^{(\bfc'',t_\lz)}_{\bfc_0}\in\cZ$.
\end{lemma}

Note that the monomials $\fkm^{\oz'_\pi}$ depends on the choice of $\oz_\pi,\pi\in\Pi^a$, and so does $\fkm^{\oz(\bfc_0)}$.

\subsection{Monomials for homogeneous regular part}

\begin{lemma}\label{monomial for P}
For $m\geq 0$, we have
$$\fkm^{m\dz}=H_m+\sum_{\bfc_\pm=0,\bfc_0\neq 0}\phi^{(\bfc,t_\lz)}_{m\dz}(v)N(\bfc,t_\lz)+\sum_{\bfc_->0,\bfc_+>0}\phi^{(\bfc,t_\lz)}_{m\dz}(v)N(\bfc,t_\lz),$$
with $\phi^{(\bfc,t_\lz)}_{m\dz}\in\cZ$.
\end{lemma}
\begin{proof}
The fact that $\fkm^{m\dz}$ is a $\cZ$-sum of all modules $M$ with dimension vector $m\dz$ proves this lemma.
\end{proof}

For a partition $\lz=(\lz_1,\lz_2,\cdots,\lz_t)$, define
$$\fkm^{\oz(t_\lz)}=\fkm^{\lz_1\dz}*\fkm^{\lz_2\dz}*\cdots*\fkm^{\lz_t\dz}.$$
When $\lz=0=t_\lz$, set $\fkm^{\oz(0)}=1$.
\begin{lemma}\label{monomial for h}
For a nonzero partition $\lz=(\lz_1,\lz_2,\cdots,\lz_t)$, we have
$$\fkm^{\oz(t_\lz)}=H_\lz+\sum_{\bfc_\pm=0,\bfc_0\neq 0}\phi^{(\bfc,t_\mu)}_{t_\lz}(v)N(\bfc,t_\mu)+\sum_{\bfc_->0,\bfc_+>0}\phi^{(\bfc,t_\mu)}_{t_\lz}(v)N(\bfc,t_\mu),$$
with $\phi^{(\bfc,t_\mu)}_{t_\lz}\in\cZ$.
\end{lemma}
\begin{proof}
By using the assertion in the proof of Lemma \ref{monomial in tube}, it suffices to show that
$$\prod_i (H_{\lz_i}+\sum_{\bfc_\pm=0,\bfc_0\neq 0}\phi^{(\bfc,t_\mu)}_{t_{\lz_i}}(v)N(\bfc,t_\mu))=H_\lz+\sum_{\bfc_\pm=0,\bfc_0\neq 0}\phi^{(\bfc,t_\mu)}_{t_\lz}(v)N(\bfc,t_\mu).$$
This is true since $xy=yx$ for any $x\in\cH^*(\cR_{\rm nh})$, $y\in\cH^*(\cR_{\rm h})$.
\end{proof}

\subsection{Monomial basis}\label{Section: Monomial basis for quiver case}

For $(\bfc,t_\lz)\in\cG^a$, we define $\oz(\bfc,t_\lz)\in\cS$ and the monomial $\fkm^{\oz(\bfc,t_\lz)}$ by letting
$$\fkm^{\oz(\bfc,t_\lz)}=\fkm^{\oz(\bfc_-)}*\fkm^{\oz(\bfc_0)}*\fkm^{\oz(t_\lz)}*\fkm^{\oz(\bfc_+)}.$$

To compute this monomial, we deal with the regular part first.
\begin{lemma}\label{monomial for regular}
For $\bfc_0\in\cG_0$ aperiodic and a nonzero partition $\lz$, we have
\begin{eqnarray*}
\fkm^{\oz(\bfc_0)}*\fkm^{\oz(t_\lz)}
&=&\lan M(\bfc_0)\ran*H_\lz+\sum_{\bfc'_0\in\cG_0\atop\bfc'_0<_G\bfc_0}\phi^{\bfc'_0}_{\bfc_0}(v)\lan M(\bfc'_0)\ran*H_\lz\\
&+&\sum_{\bfc''_\pm=0,|\lz''|<|\lz|}\phi^{(\bfc'',t_{\lz''})}_{\bfc_0,t_\lz}(v)N(\bfc'',t_{\lz''})  \\
&+&\sum_{\bfc'''_->0,\bfc'''_+>0}\phi^{(\bfc''',t_\mu)}_{\bfc_0,t_\lz}(v)N(\bfc''',t_\mu)
\end{eqnarray*}
\end{lemma}
\begin{proof}
By Lemma \ref{monomial for nh}, Lemma \ref{monomial for h}, $\fkm^{\oz(\bfc_0)}*\fkm^{\oz(t_\lz)}$ can be written as a $\cZ$-sum of terms of following kinds:
\begin{enumerate}
  \item [(1)] The leading term $\lan M(\bfc_0)\ran*H_\lz$.
  \item [(2)] $\sum_{\bfc'_0\in\cG_0\atop\bfc'_0<_G\bfc_0}\phi^{\bfc'_0}_{\bfc_0}(v)\lan M(\bfc'_0)\ran*H_\lz$.
  \item [(3)] $\lan M(\bfc'_0)\ran*N(\bfc'',t_{\lz''})$ with $\bfc'_0\leq_G\bfc_0$, $\bfc''_\pm=0,\bfc''_0\neq 0$. Since $D(\bfc'',t_{\lz''})=|\lz|\dz$,  $|\lz''|<|\lz|$. Also, $\lan M(\bfc'_0)\ran*\lan M(\bfc''_0)\ran\in\cH^*(\cR_{\rm nh})$, therefore $\lan M(\bfc'_0)\ran*N(\bfc'',t_{\lz''})$ is a $\cZ$-sum of $N(\bfc''',t_{\lz'''})$ with  $\bfc'''_\pm=0,|\lz'''|<|\lz|$.
  \item [(4)] $N(\bfc^1,t_{\lz^1})*N(\bfc^2,t_{\lz^2})$ with $\bfc^1_->0,\bfc^1_+>0$ or $\bfc^2_->0,\bfc^2_+>0$. By Proposition \ref{multiofN} and Lemma \ref{defect}, this is again a $\cZ$-sum of $N(\bfc^3,t_{\lz^3})$ with $\bfc^3_->0,\bfc^3_+>0$.
\end{enumerate}
Hence we get the desired result.
\end{proof}

Recall from Section \ref{def of N(c,lz)} that by definition, 
$$N(\bfc,T_\lz)=\lan M(\bfc_-)\ran*\lan M(\bfc_0)\ran*H_\lz*\lan M(\bfc_+)\ran=\sum_{\mu\geq\lz} K_{\lz\mu} N(\bfc,t_\mu).$$
\begin{proposition}
For $(\bfc,t_\lz)\in\cG^a$, we have
\begin{align}
\fkm^{\oz(\bfc,t_\lz)}&=N(\bfc,T_\lz)+\sum_{(\bfc',t_{\lz'})\in\cG\atop(\bfc',|\lz'|)\prec(\bfc,|\lz|)}\phi_{(\bfc,t_\lz)}^{(\bfc',t_{\lz'})}(v)N(\bfc',t_{\lz'})\notag\\
&=\sum_{\mu\geq\lz} K_{\lz\mu} N(\bfc,t_\mu)+\sum_{(\bfc',t_{\lz'})\in\cG\atop(\bfc',|\lz'|)\prec(\bfc,|\lz|)}\phi_{(\bfc,t_\lz)}^{(\bfc',t_{\lz'})}(v)N(\bfc',t_{\lz'}),\notag
\end{align}
with $\phi_{(\bfc,t_\lz)}^{(\bfc',t_{\lz'})}\in\cZ$, where $K_{\lz\mu}$ is the Kostka number.
\end{proposition}

\begin{proof}

By Lemma \ref{momomial for proj}, Lemma \ref{monomial for inj}, Lemma \ref{monomial for regular}, $\fkm^{\oz(\bfc,t_\lz)}$ can be written as a $\cZ$-sum of terms of following kinds:
\begin{enumerate}
  \item [(1)] The leading term $\lan M(\bfc_-)\ran*\lan M(\bfc_0)\ran*H_\lz*\lan M(\bfc_+)\ran=N(\bfc,T_\lz)$ with coefficient $1$.
  \item [(2)] $N(\bfc^1,t_{\lz^1})*N(\bfc^2,t_{\lz^2})*N(\bfc^3,t_{\lz^3})$ with $\bfc^1_->_L\bfc_-$ or $\bfc^3_+>_L\bfc_+$. By Proposition \ref{multiofN} and Definition \ref{order}, this is a $\cZ$-sum of $N(\bfc',t_{\lz'})$ with $\bfc'_\pm>_L\bfc_\pm$.

  \item [(3)] $\lan M(\bfc_-)\ran*\lan M(\bfc'_0)\ran*H_\lz*\lan M(\bfc_+)\ran$ with $\bfc'_0<_G\bfc_0$.
  \item [(4)] $\lan M(\bfc_-)\ran*\lan M(\bfc''_0)\ran*H_{\lz''}*\lan M(\bfc_+)\ran$ with $|\lz''|<|\lz|$.
  \item [(5)] $\lan M(\bfc_-)\ran*N(\bfc''',t_{\mu})*\lan M(\bfc_+)\ran$ with $\bfc'''_->0,\bfc'''_+>0$. By Corollary \ref{preprojdecomp} and Corollary \ref{preinjdecomp}, this is a $\cZ$-sum of $N(\bfc',t_{\lz'})$ with $\bfc'_->_L\bfc_-$ and $\bfc'_+>_L\bfc_+$.
\end{enumerate}
By Definition \ref{order}, the proposition follows.

\end{proof}

Since $(K_{\lz\mu})_{\lz,\mu}$ is upper-triangular with entries in the diagonal equal to $1$, we have the following corollary.
\begin{corollary}
For $(\bfc,t_\lz)\in\cG^a$, we have
$$\fkm^{\oz(\bfc,t_\lz)}=N(\bfc,t_\lz)+\sum_{(\bfc',t_{\lz'})\in\cG\atop(\bfc',t_{\lz'})\prec(\bfc,t_\lz)}\phi_{(\bfc,t_\lz)}^{(\bfc',t_{\lz'})}(v)N(\bfc',t_{\lz'}),$$
with $\phi_{(\bfc,t_\lz)}^{(\bfc',t_{\lz'})}\in\cZ$, such that $\phi_{(\bfc,t_\lz)}^{(\bfc,t_\mu)}=K_{\lz\mu},\mu\geq\lz$.
\end{corollary}

\subsection{PBW bases}
We now use the same method in Section \ref{PBW basis of cyclic quiver} to construct a PBW basis $\{E(\bfc,t_\lz)|(\bfc,t_\lz)\in\cG^a\}$ for $\cC^*_\cZ$ inductively.

\begin{definition}\label{def of E}
For any $\nu\in\bbN I$, if $(\bfc,t_\lz)$ is minimial in $\cG^a_\nu$, put
$$E(\bfc,t_\lz)=\fkm^{\oz(\bfc,t_\lz)}\in\cC^*_\cZ.$$
In fact, $(\bfc,t_\lz)$ is minimal if and only if $M(\bfc)$ is semisimple and $t_\lz=0$.

Assume that $E(\bfc',t_{\lz'})\in\cC^*_\cZ$ have been defined for all $(\bfc',t_{\lz'})\in\cG^a_\nu$ with $(\bfc',t_{\lz'})\prec(\bfc,t_\lz)$.
Then we define
$$E(\bfc,t_\lz)=\fkm^{\oz(\bfc,t_\lz)}-\sum_{(\bfc',t_{\lz'})\in\cG^a_\nu\atop (\bfc',t_{\lz'})\prec(\bfc,t_\lz)}\phi_{(\bfc,t_\lz)}^{(\bfc',t_{\lz'})}(v)E(\bfc',t_{\lz'})\in\cC^*_\cZ.$$
In other words, we always have
$$\fkm^{\oz(\bfc,t_\lz)}=E(\bfc,t_\lz)+\sum_{(\bfc',t_{\lz'})\in\cG^a_\nu\atop(\bfc',t_{\lz'})\prec(\bfc,t_\lz)}\phi_{(\bfc,t_\lz)}^{(\bfc',t_{\lz'})}(v)E(\bfc',t_{\lz'})$$
for any $(\bfc,t_\lz)\in\cG^a$.
\end{definition}

By construction, we have the following lemma.

\begin{lemma}\label{E in N}
For $(\bfc,t_\lz)$ in $\cG^a$, we have
$$E(\bfc,t_\lz)=N(\bfc,t_\lz)+\sum_{(\bfc',t_{\lz'})\in\cG_\nu\setminus\cG^a_\nu\atop (\bfc',t_{\lz'})\prec(\bfc,t_\lz)}b_{(\bfc,t_\lz)}^{(\bfc',t_{\lz'})}(v)N(\bfc',t_{\lz'}),$$
with $b_{(\bfc,t_\lz)}^{(\bfc',t_{\lz'})}\in\cZ$.
\end{lemma}

This shows that the leading term of $E(\bfc,t_\lz)$ is $N(\bfc,t_\lz)$, and this is why we say $\{E(\bfc,t_\lz)|(\bfc,t_\lz)\in\cG^a\}$ is of PBW-type.
Next we prove that $\{E(\bfc,t_\lz)|(\bfc,t_\lz)\in\cG^a\}$ is a $\cZ$-basis of $\cC^*_\cZ$.
\begin{lemma}\label{E is Qbasis}
The set $\{E(\bfc,t_\lz)|(\bfc,t_\lz)\in\cG^a\}$ is a $\bbQ(v)$-basis of $\cC^*$.
\end{lemma}
\begin{proof}
By Lemma \ref{E in N}, $\{E(\bfc,t_\lz)|(\bfc,t_\lz)\in\cG^a\}$ is linearly independent.
By Theorem 4.16 in \cite{Lusztig_Affine_quivers_and_canonical_bases}, $|\cG^a_\nu|=\dime\bff_\nu$.
\end{proof}
\begin{lemma}\label{periodic not in C}
Let $\mathbf{P}$ be the $\bbQ(v)$-linear subspace of $\cH^0$ spanned by $\{N(\bfc,t_\lz)|(\bfc,t_\lz)\in\cG\setminus\cG^a\}$. Then $\mathbf{P}\cap\cC^*=\{0\}$.
\end{lemma}
\begin{proof}
This is a direct result from Lemma \ref{Lemma: tN is basis}, Lemma \ref{E in N} and Lemma \ref{E is Qbasis}.
\end{proof}
\begin{corollary}
$\{E(\bfc,t_\lz)|(\bfc,t_\lz)\in\cG^a\}$ is independent of the choice of $\oz_\pi,\pi\in\Pi^a$ for each non-homogeneous tube.
\end{corollary}
\begin{proof}
Suppose we used another choice of $\oz_\pi,\pi\in\Pi^a$ for each non-homogeneous tube to construct a new basis $\{E'(\bfc,t_\lz)|(\bfc,t_\lz)\in\cG^a\}$ by the same method.
Then it should have the same property as in Lemma \ref{E in N}.
Hence, $E'(\bfc,t_\lz)-E(\bfc,t_\lz)\in\mathbf{P}\cap\cC^*$.
By Lemma \ref{periodic not in C}, we have
$E'(\bfc,t_\lz)=E(\bfc,t_\lz).$
\end{proof}
\begin{proposition}\label{E is integral basis}
The set $\{E(\bfc,t_\lz)|(\bfc,t_\lz)\in\cG^a\}$ is a $\cZ$-basis of $\cC^*_\cZ$.
\end{proposition}
\begin{proof}
By Proposition \ref{multiofN}, for any monomial $\fkm^\oz$ in $\cC^*_\cZ$,
$$\fkm^\oz=\sum_{(\bfc,t_\lz)\in\cG}\psi^{(\bfc,t_\lz)}_\oz(v)N(\bfc,t_\lz)$$
with $\psi^{(\bfc,t_\lz)}_\oz\in\cZ$. Then
$$\fkm^\oz-\sum_{(\bfc,t_\lz)\in\cG^a}\psi^{(\bfc,t_\lz)}_\oz(v)E(\bfc,t_\lz)\in\mathbf{P}\cap\cC^*$$
implies that $$\fkm^\oz=\sum_{(\bfc,t_\lz)\in\cG^a}\psi^{(\bfc,t_\lz)}_\oz(v)E(\bfc,t_\lz),$$
which means that any monomial is a $\cZ$-linear combination of elements in $\{E(\bfc,t_\lz)|(\bfc,t_\lz)\in\cG^a\}$.
\end{proof}
\begin{corollary}
The set $\{\fkm^{\oz(\bfc,t_\lz)}|(\bfc,t_\lz)\in\cG^a\}$ is a $\cZ$-basis of $\cC^*_\cZ$.
\end{corollary}

We say that $\{E(\bfc,t_\lz)|(\bfc,t_\lz)\in\cG^a\}$ is the PBW basis and  $\{\fkm^{\oz(\bfc,t_\lz)}|(\bfc,t_\lz)\in\cG^a\}$ is a monomial basis of $\cC^*_\cZ$.

Moreover, let  $\mathbf{P}_\cZ$ be the $\cZ$-submodule of $\cH^0_\cZ$ spanned by $\{N(\bfc,t_\lz)|(\bfc,t_\lz)\in\cG\setminus\cG^a\}$, then by Lemma \ref{E in N}, Lemma \ref{periodic not in C} and proposition \ref{E is integral basis}, we have
$$\cH^0_\cZ=\cC^*_\cZ\oplus\mathbf{P}_\cZ$$
as $\cZ$-modules and
$$E(\bfc,t_\lz)=N(\bfc,t_\lz) \,\mod \mathbf{P}_\cZ$$
for $(\bfc,t_\lz)\in\cG^a$.

\subsection{A bar-invariant basis}\label{Section: construction of C(c,lz)}
By Definition \ref{def of E}, the transition matrix between the monomial basis $\{\fkm^{\oz(\bfc,t_\lz)}|(\bfc,t_\lz)\in\cG^a\}$ and the PBW basis $\{E(\bfc,t_\lz)|(\bfc,t_\lz)\in\cG^a\}$ is uppertriangular in $\cZ$ with entries in the diagonal equal to $1$. So we have
$$E(\bfc,t_\lz)=\fkm^{\oz(\bfc,t_\lz)}+\sum_{(\bfc',t_{\lz'})\in\cG^a\atop(\bfc',t_{\lz'})\prec(\bfc,t_\lz)}\eta_{(\bfc,t_\lz)}^{(\bfc',t_{\lz'})}(v)\fkm^{\oz(\bfc',t_{\lz'})}$$
for $(\bfc,t_\lz)\in\cG^a$, with $\eta_{(\bfc,t_\lz)}^{(\bfc',t_{\lz'})}\in\cZ$.

Recall from Section \ref{Section: The bar involution and the canonical basis} that the bar involution $\bar{\cdot}$ is a $\bbQ$-algebra automorphism on $\bff\cong\cC^*$, and clearly all monomials are bar-invariant.
Applying the bar involution to the both sides, then we get
$$\overline{E(\bfc,t_\lz)}=\fkm^{\oz(\bfc,t_\lz)}+\sum_{(\bfc',t_{\lz'})\in\cG^a\atop(\bfc',t_{\lz'})\prec(\bfc,t_\lz)}\overline{\eta}_{(\bfc,t_\lz)}^{(\bfc',t_{\lz'})}(v)\fkm^{\oz(\bfc',t_{\lz'})}$$
for $(\bfc,t_{\lz})\in\cG^a$.
Thus we have
$$\overline{E(\bfc,t_\lz)}=E(\bfc,t_\lz)+\sum_{(\bfc',t_{\lz'})\in\cG^a\atop(\bfc',t_{\lz'})\prec(\bfc,t_\lz)}\zeta_{(\bfc,t_\lz)}^{(\bfc',t_{\lz'})}(v)E(\bfc',t_{\lz'})$$
for $(\bfc,t_{\lz})\in\cG^a$, with $\zeta_{(\bfc,t_\lz)}^{(\bfc',t_{\lz'})}\in\cZ$.

By Section 7.10 in \cite{Lusztig_Canonical_bases_arising_from_quantized_enveloping_algebra}, the system
$$g^{(\bfc',t_{\lz'})}_{(\bfc,t_{\lz})}=\sum_{(\bfc'',t_{\lz''})\in\cG^a,(\bfc',t_{\lz'})\preceq(\bfc'',t_{\lz''})\preceq(\bfc,t_\lz)}\zeta^{(\bfc',t_{\lz'})}_{(\bfc'',t_{\lz''})}\overline{g}^{(\bfc'',t_{\lz''})}_{(\bfc,t_\lz)}$$ for $(\bfc',t_{\lz'})\preceq(\bfc,t_\lz)\in\cG^a$,
has a uinque solution satisfying $g^{(\bfc,t_\lz)}_{(\bfc,t_\lz)}=1,g^{(\bfc',t_{\lz'})}_{(\bfc,t_\lz)}\in v^{-1}\bbZ[v^{-1}]$ for $(\bfc',t_{\lz'})\preceq(\bfc,t_\lz)$.
Let
$$C(\bfc,t_\lz)=E(\bfc,t_\lz)+\sum_{(\bfc',t_{\lz'})\in\cG^a,(\bfc',t_{\lz'})\prec(\bfc,t_\lz)}g^{(\bfc',t_{\lz'})}_{(\bfc,t_\lz)}(v)E(\bfc',t_{\lz'}).$$

Therefore by construction, we have
$$\overline{C{(\bfc,t_\lambda)}}=C{(\bfc,t_\lambda)}$$
for all $(\bfc,t_\lz)\in\cG^a$. We call the set $\{C{(\bfc,t_\lambda)}|(\bfc,t_\lambda)\in\mathcal{G}^a\}$ the bar-invariant basis.
Moreover, we have an important relation between the monomial basis and the bar-invariant basis:
\begin{equation}\label{eqn: monomial and bar inv}
\fkm^{\oz(\bfc,t_\lz)}=C(\bfc,t_\lz)+\sum_{(\bfc',t_{\lz'})\in\cG^a\atop(\bfc',t_{\lz'})\preceq(\bfc,t_\lz)}\tilde{\az}^{(\bfc',t_{\lz'})}_{(\bfc,t_\lz)}(v)C(\bfc',t_{\lz'}).
\end{equation}

\begin{proposition}\label{proposition: E satisfy three conditions}
The set $\{E{(\bfc,t_\lambda)}|(\bfc,t_\lambda)\in\mathcal{G}^a\}$ is a $\cZ$-basis of $\mathcal{C}^{\ast}(Q)_{\cZ}$
satisfying the following conditions.
\begin{enumerate}
  \item[(1)]$\overline{E(\bfc,t_\lz)}=E(\bfc,t_\lz)+\sum_{(\bfc',t_{\lz'})\in\cG^a\atop(\bfc',t_{\lz'})\prec(\bfc,t_\lz)}\zeta_{(\bfc,t_\lz)}^{(\bfc',t_{\lz'})}(v)E(\bfc',t_{\lz'}),\zeta_{(\bfc,t_\lz)}^{(\bfc',t_{\lz'})}\in\cZ.$
  \item[(2)]$(E{(\bfc,t_\lambda)},E{(\bfc',t_{\lambda'})})\in\delta_{(\bfc,t_\lambda),(\bfc',t_{\lambda'})}+v^{-1}\mathbb{Q}[[v^{-1}]]\cap\mathbb{Q}(v)$.
\end{enumerate}
\end{proposition}
\begin{theorem}\label{Theorem: C satisfy three conditions}
The set $\{C{(\bfc,t_\lambda)}|(\bfc,t_\lambda)\in\mathcal{G}^a\}$ is a $\cZ$-basis of $\mathcal{C}^{\ast}(Q)_{\cZ}$
satisfying the following conditions.
\begin{enumerate}
  \item[(1)]$\overline{C{(\bfc,t_\lambda)}}=C{(\bfc,t_\lambda)}$;
  \item[(2)]$(C{(\bfc,t_\lambda)},C{(\bfc',t_{\lambda'})})\in\delta_{(\bfc,t_\lambda),(\bfc',t_{\lambda'})}+v^{-1}\mathbb{Q}[[v^{-1}]]\cap\mathbb{Q}(v)$.
\end{enumerate}
\end{theorem}

The condition (1) of Proposition \ref{proposition: E satisfy three conditions} and Theorem \ref{Theorem: C satisfy three conditions} are proved by the construction above, we leave the rest of the proof in Section \ref{Chapter: coincidence of two bases} by showing that $\{C(\bfc,t_\lz)|(\bfc,t_\lz)\in\cG^a\}$ is in fact the canonical basis $\bfB$ of $\bff$.

\section{Lusztig's geometric construction and the canonical basis}
We give a review of Lusztig's geometric construction of the canonical basis. We refer to \cite{Lusztig_Introduction_to_quantum_groups} for basic constructions and \cite{Li_Notes_on_affine_canonical_and_monomial_bases} for more details about affine canonical basis.
\subsection{Representation spaces and flag varieties}
Let $V=\oplus_{i\in I} V_i$ be an $I$-graded vector space over an algebraic closed field $k$.
Denote $\udim V=\sum_{i\in I}\dime V_ii\in\bbN I$.
Similarly to Section \ref{Geometric property of the order}, we define
$$\bbE_V=\oplus_{h\in H}\Hom(V_{s(h)},V_{t(h)}), \GL_V=\oplus_{i\in I}\GL(V_i).$$
The $\GL_V$-action on $\bbE_V$ is also defined by
$$g.x=((g.x)_h)_{h\in H},(g.x)_h=g_{t(h)}x_hg^{-1}_{s(h)}$$
for any $g=(g_i)_{i\in I}\in\GL_V,x=(x_h)_{h\in H}\in\bbE_V$.

Given any $\nu\in\bbN I$, let $\cS_{\nu}$ be the set consisting of all pairs of sequences defined in Section \ref{Section: def of monomial}:
$$\oz=(\underline{i},\underline{a})=((i_1,\dots,i_t),(a_1,\dots,a_t))$$
such that $\sum_ma_mi_m=\nu$.

Fix an $I$-graded vector space $V$ such that $\udim V=\nu$. For any $\oz\in\cS_{\nu}$, we say that a flag of subspaces
$$V^\bullet=(V=V^0\supseteq V^1\supseteq\cdots\supseteq V^t=0)$$
is of type $\oz$ if $\udim V^{m-1}/V^m=a_mi_m$ for all $m=1,2,\cdots,t$.

Let $\cF_\oz$ be the flag variety consisting of all flags of type $\oz$. Then $\GL_V$ acts on $\cF_\oz$ naturally.

Given $x\in\bbE_V, V^\bullet\in\cF_\oz$, we say that $V^\bullet$ is $x$-stable if $x_h(V^m_{s(h)})\subseteq V^m_{t(h)}$ for any $h\in H$ and $m=1,2,\cdots,t$.
We define $\tcF_\oz$ to be the variety consisting of all pairs $(x,V^\bullet)\in\bbE_V\times\cF_\oz$ such that $V^\bullet$ is $x$-stable. So $\GL_V$ acts on $\tcF_\oz$ naturally.

\subsection{Perverse sheaves}

Fix a prime $l$ that is invertible in $k$.
Let $\overline{\bbQ}_l$ be the algebraic closure of the field of $l$-adic numbers.
Given any algebraic variety $X$ over $k$,
denote by $\cD(X)$ the bounded derived category of complexes of $l$-adic sheaves on $X$.
Let $\cM(X)$ be the full subcategory of $\cD(X)$ consisting of all perverse sheaves on $X$.

Let $G$ be a connected algebraic group. Assume that $G$ acts on $X$ algebraically.
By $\cD_G(X)$ we denote the full subcategory of $\cD(X)$ consisting of all $G$-equivariant complexes over $X$.
By $\cM_G(X)$ we denote the full subcategory of $\cM(X)$ consisting of all $G$-equivariant perverse sheaves over $X$.

Let $\mathbf{1}_{\tcF_\oz}$ be the constant sheaf on $\tcF_\oz$.
Let $\pi_\oz:\tcF_\oz\ra\bbE_V$ be the first projection $(x,V^\bullet)\mapsto x$.
By the decomposition theorem of Beilinson, Bernstein and Deligne (\cite{BBD_Faisceaux_pervers}), the Lusztig sheaf
$$L_\oz:=(\pi_\oz)_!\mathbf{1}_{\tcF_\oz}[d_\oz](\frac{d_\oz}{2})$$
is a $\GL_V$-equivariant semisimple complex on $\bbE_V$, where $[d]$ is the shift by $d$ of complexes and $(\frac{d}{2})$ is the Tate twist, $d_\oz=\dime \tcF_\oz$.

Let $\cP_V$ be the subcategory of $\cM_{\GL_V}(\bbE_V)$ consisting of direct sums of simple perverse sheaves, which are direct summands of $L_\oz[r](\frac{r}{2})$ for some $\oz\in\cS_{\udim V}$ and $r\in\bbZ$.

Let $\cQ_V$ be the full subcategory of $\cD_{\GL_V}(\bbE_V)$ whose objects are the complexes that are isomorphic to finite direct sums of complexes of the form $L[d](\frac{d}{2})$ for various $L\in\cP_V$ and $d\in\bbZ$.

Let $K_{\udim V}=K_V=K(\cQ_V)$ be the Grothendieck group of $\cQ_V$. Define $$(-v)[L]=[L[1](\frac{1}{2})],(-v^{-1})[L]=[L[-1](-\frac{1}{2})].$$
Then $K_V$ is a free $\cZ$-module, where $\cZ=\bbZ[v,v^{-1}]$. Define
$$K(\cQ)=\bigoplus_{\nu\in\bbN I}K_\nu.$$

\subsection{Lusztig's induction functors}
Let $W\subset V$ be an $I$-graded subspace. Let $T=V/W$ and $p:V\ra T$ be the natural projection.

Given any $x\in\bbE_V$, $W$ is called $x$-stable if $x_h(W_{s(h)})\subseteq W_{t(h)}$ for all $h\in H$.

If $W$ is $x$-stable, it induces two elements $x_W$ and $x_T$ in $\bbE_W$ and $\bbE_T$ respectively such that $(x_W)_h$ is the restriction of $x_h$ to $W$ and $p_{t(h)}x_h=(x_T)_h p_{s(h)}$ for all $h\in H$.

We consider the following diagram
$$\bbE_T\times\bbE_W\stackrel{q_1}{\longleftarrow}\bbE'\stackrel{q_2}{\lra}\bbE''\stackrel{q_3}{\lra}\bbE_V,$$
where $\bbE''=\{(x,V')|V'$ is $x$-stable, $\udim V'=\udim W\}$; $\bbE'$ is the variety consisting of all quadruples $(x,V',r',r'')$ such that $(x,V')\in\bbE''$, $r':V/V'\ra T$ and $r'':V'\ra W$ are graded linear isomorphisms; $q_2,q_3$ are natural projections and $q_1(x,V',r',r'')=(y',y'')$ such that $y'_h=r'_{t(h)}(x_{V/V'})_h(r')^{-1}_{s(h)}$ and $y''_h=r''_{t(h)}(x_{V'})_h(r'')^{-1}_{s(h)}$ for all $h\in H$.

Given any $L'\in\cQ_T,L''\in\cQ_W$, we set
$$L'\circ L''=(q_3)_!(q_2)_\flat q_1^*(L'\boxtimes L'')[d_1-d_2](\frac{d_1-d_2}{2}),$$
where $q_2^*$ is an equivalence and $(q_2)_\flat$ is the inverse of $q_2^*$, $d_1$ is the dimension of the fibres of $q_1$ and $d_2$ is the dimension of the fibres of $q_2$.

The operation $\circ$ induces a multiplication $\circ$ on $K(\cQ)$ such that $[L']\circ[L'']=[L'\circ L'']$. This makes $K(\cQ)$ a $\cZ$-algebra.

Lusztig shows that $K(\cQ)$ is isomorphic to $\bff_\cZ$ by mapping $L_{\oz}$ to $\fkm^\oz$ (Chapter 13, \cite{Lusztig_Introduction_to_quantum_groups}).

\subsection{The geometric constructions of monomial and canonical basis}
The monomial basis $\{F(\bfa,\bm{\lz})|(\bfa,\bm{\lz})\in\Delta\}$ given in \cite{Li_Notes_on_affine_canonical_and_monomial_bases} coincides the monomial basis $\{\fkm^{\oz(\bfc,t_\lz)}|(\bfc,t_\lz)\in\cG^a\}$ constructed in Section \ref{Section: Monomial basis for quiver case} under the identification $K(\cQ)\cong\bff_\cZ$.

By \cite{Li_Notes_on_affine_canonical_and_monomial_bases}, we have
\begin{equation}\label{equation: monomial and canonical in sheaves}
F(\bfa,\bm{\lz})=L_{\oz(\bfc,t_\lz)}=\oplus_{\mu\geq\lz} K_{\lz\mu}\IC(X(\bfc,|\lz|),\cL_\mu) \oplus B,
\end{equation}
where $\IC(X(\bfc,m),\cL_\mu)$ is the intersection cohomology sheaf with support $\overline{X(\bfc,m)}$ and local system $\cL_\mu$ such that $\mu\vdash m$, and ${\rm supp\,}B\subset \overline{X(\bfc,|\lz|)}\setminus X(\bfc,|\lz|)$.

It is also proved in \cite{Li_Notes_on_affine_canonical_and_monomial_bases} that any simple perverse sheaf in $\cP_V$ is isomorphic to $\IC(X(\bfc,|\lz|),\cL_\lz)$ for some $(\bfc,t_\lz)\in\cG^a_{\udim V}$.
Let $\bfB_{IC}$ be the set of all $B(\bfc,t_\lz)=[\IC(X(\bfc,|\lz|),\cL_\lz)]$ in $K(\cQ)$ for various $(\bfc,t_\lz)\in\cG^a$. Then we have the following lemma.

\begin{lemma}[\cite{Li_Notes_on_affine_canonical_and_monomial_bases}]
Under the identification $\bff_\cZ\cong K(\cQ)$, we have $\bfB=\bfB_{IC}$.
\end{lemma}
\begin{proof}
See Proposition 5.3 in \cite{Li_Notes_on_affine_canonical_and_monomial_bases}.
\end{proof}

So the image of equation (\ref{equation: monomial and canonical in sheaves}) in the Grothendieck group $K(\cQ)$ is
\begin{equation}\label{equation: monomial and canonical in bff}
\fkm^{\oz(\bfc,t_\lz)}=\sum_{\mu\geq\lz}K_{\lz\mu}B(\bfc,t_\mu)+\sum_{(\bfc',t_{\lz'})\in\cG^a\atop\overline{X(\bfc',|\lz'|)}\subset\overline{X(\bfc,|\lz|)}\setminus X(\bfc,|\lz|)}\az^{(\bfc',t_{\lz'})}_{(\bfc,t_\lz)}B(\bfc',t_{\lz'}).
\end{equation}

\subsection{Trace map}\label{Section: trace map}
For a perverse sheaf $L$ in $\cP_V$, the specialization $L\mapsto\Phi_q([L])$ is the trace map introduced in \cite{Kiehl_Weissauer_Weil_conjectures_perverse_sheaves_and_l'adic_Fourier_transform}.
Also in \cite{Lusztig_Canonical_bases_and_Hall_algebras}, Lusztig gives a way to compute $\Phi_q([L])\in\cH^*_q$ as follows.

Let the ground field $k=\overline{\bbF}_q$. An $\bbF_q$-linear isomorphism $\fkf:V\ra V$ on a $k$-vector space $V$ is called a Frobenius map if it satisfies the following conditions.
\begin{enumerate}
\item[(1)] $\fkf(yv)=y^q \fkf(v)$, $v\in V$, $y\in k$.
\item[(2)] For any $v\in V$, $\fkf^t(v)=v$ for some $t\geq 1$.
\end{enumerate}
Given a Frobenius map $\fkf:V\ra V$, $V^{\fkf}=\{v\in V|\fkf(v)=v\}$ is naturally a $\bbF_q$-vector space.

Now let $V=\oplus_{i\in I}V_i$ be an $I$-graded $k$-vector space and $\fkf:V\ra V$ be a Frobenius map such that $\fkf(V_i)=V_i$.
Then there is a naturally induced Frobenius map $\fkf:\bbE_V\ra \bbE_V$ such that $\fkf(x)_h(\fkf(v))=\fkf(x_h(v))$ for any $x\in\bbE_V, h\in H$ and $v\in V_{s(h)}$.

For a perverse sheaf $L$ in $\cP_V$, $\fkf:\bbE_V\ra \bbE_V$ induces a canonical isomorphism $\fkf^*L\cong L$.
This induces a linear map $H^j(L)|_x=H^j(L)|_{\fkf(x)}=H^j(\fkf^*L)|_x\cong H^j(L)|_x$ for $x\in\bbE_V^{\fkf}$.
Taking the trace of this linear map times $(-1)^j$ and summing over $j$, we obtain an element $\chi_{L,q}(x)\in\overline{\bbQ}_l$.
We thus have a function $\chi_{L,q}:\bbE_V^{\fkf}\ra\overline{\bbQ}_l$.

Note that for any $x\in\bbE_V^{\fkf}$, $(V^{\fkf},x)$ is a representation of $Q$ over $\bbF_q$, $\chi_{L,q}$ can be regarded as an element in $\cH^*_q$ by identifying the characteristic function $\chi_{[M]}$ of the orbit of $M$ with $\lan M\ran$ in $\cH^*_q$. Under the field isomorphism $\overline{\bbQ}_l\cong\bbC$, Lusztig showed that $\chi_{L,q}=\Phi_q([L])$.

\section{Coincidence of the canonical basis and bar-invariant basis}\label{Chapter: coincidence of two bases}
\subsection{}
In this section, we always identify $K(\cQ)$, $\bff_\cZ$, $\bfU^+_\cZ$ and $\cC^*_\cZ$.

\begin{lemma}
For $(\bfc,t_\lz)\in\cG^a$, we have
\begin{equation}\label{Equation: B in N}
B(\bfc,t_\lz)=N(\bfc,t_\lz)+\sum_{(\bfc',t_{\lz'})\in\cG\atop(\bfc',t_{\lz'})\prec(\bfc,t_\lz)}a^{(\bfc',t_{\lz'})}_{(\bfc,t_\lz)}(v)N(\bfc',t_{\lz'})
\end{equation}
with $a^{(\bfc',t_{\lz'})}_{(\bfc,t_\lz)}\in v^{-1}\bbZ[v^{-1}]$.
\end{lemma}
\begin{proof}
By Section \ref{Section: trace map}, when specialized to $k=\bbF_q$, $\Phi_q(B(\bfc,t_\lz))$ is always a $\bbQ(v_q)$-sum of some $N(\bfc',t_{\lz'})$ with $Y(\bfc',|\lz'|)\cap\overline{X(\bfc,|\lz|)}\neq\emptyset$. Therefore, by Proposition \ref{Lemma: geometric property of the order}, we have $(\bfc',|\lz'|)\preceq(\bfc,|\lz|)$. That is, for any $q$, the coefficient of $N(\bfc',t_{\lz'})$ with $(\bfc',|\lz'|)\npreceq(\bfc,|\lz|)$ is always zero.
So as an element in $\cH^0_\cZ$, $B(\bfc,t_\lz)$ is a $\cZ$-sum of some $N(\bfc',t_{\lz'})$ with $(\bfc',|\lz'|)\prec(\bfc,|\lz|)$.

Therefore by equation (\ref{equation: monomial and canonical in bff}),
$$\fkm^{\oz(\bfc,t_\lz)}-\sum_{\mu\geq\lz}K_{\lz\mu}B(\bfc,t_\mu)\in\cZ-{\rm span}\{N(\bfc',t_{\lz'})|(\bfc',t_{\lz'})\in\cG,(\bfc',|\lz'|)\prec(\bfc,|\lz|)\}.$$

Recall that
$$\fkm^{\oz(\bfc,t_\lz)}=\sum_{\mu\geq\lz} K_{\lz\mu} N(\bfc,t_\mu)+\sum_{(\bfc',t_{\lz'})\in\cG\atop(\bfc',|\lz'|)\prec(\bfc,|\lz|)}\phi_{(\bfc,t_\lz)}^{(\bfc',t_{\lz'})}(v)N(\bfc',t_{\lz'}).$$
So for all $(\bfc,t_\lz)\in\cG^a$, we have
$$B(\bfc,t_\lz)=N(\bfc,t_\lz)+\sum_{(\bfc',t_{\lz'})\in\cG\atop(\bfc',t_{\lz'})\prec(\bfc,t_\lz)}a^{(\bfc',t_{\lz'})}_{(\bfc,t_\lz)}(v)N(\bfc',t_{\lz'}),$$
with $a^{(\bfc',t_{\lz'})}_{(\bfc,t_\lz)}\in\cZ$.

By the almost orthogonality of $\{B(\bfc,t_\lz)|(\bfc,t_\lz)\in\cG^a\}$ and $\{N(\bfc,t_\lz)|(\bfc,t_\lz)\in\cG\}$ (see Section \ref{Section: Almost orthogonality of N}), we have $a^{(\bfc',t_{\lz'})}_{(\bfc,t_\lz)}\in v^{-1}\bbZ[v^{-1}]$.

\end{proof}

\begin{corollary}\label{Corollary: E in B}
For $(\bfc,t_\lz)\in\cG^a$, we have
$$E(\bfc,t_\lz)=B(\bfc,t_\lz)+\sum_{(\bfc',t_{\lz'})\in\cG^a\atop(\bfc',t_{\lz'})\prec(\bfc,t_\lz)}\ta_{(\bfc,t_\lz)}^{(\bfc',t_{\lz'})}(v)B(\bfc',t_{\lz'}),$$ with $\ta_{(\bfc,t_\lz)}^{(\bfc',t_{\lz'})}\in v^{-1}\bbZ[v^{-1}]$.
\end{corollary}
\begin{proof}
By doing the same operation in Definition \ref{def of E} on equations (\ref{Equation: B in N}) for various $(\bfc,t_\lz)\in\cG^a$, we can inductively get

$$
\tilde{E}(\bfc,t_\lz)=B(\bfc,t_\lz)+\sum_{(\bfc',t_{\lz'})\in\cG^a\atop(\bfc',t_{\lz'})\prec(\bfc,t_\lz)}\ta_{(\bfc,t_\lz)}^{(\bfc',t_{\lz'})}(v)B(\bfc',t_{\lz'})
$$
for some $\ta_{(\bfc,t_\lz)}^{(\bfc',t_{\lz'})}\in v^{-1}\bbZ[v^{-1}]$, such that $\{\tilde{E}(\bfc,t_\lz)|(\bfc,t_\lz)\in\cG^a\}$ is a $Q(v)$-basis of $\cC^*$, and
\begin{equation}\label{Equation: operation on B in N}
\tilde{E}(\bfc,t_\lz)=N(\bfc,t_\lz)+\sum_{(\bfc',t_{\lz'})\in\cG\setminus\cG^a\atop(\bfc',t_{\lz'})\prec(\bfc,t_\lz)}f_{(\bfc,t_\lz)}^{(\bfc',t_{\lz'})}(v)N(\bfc',t_{\lz'}),
\end{equation}
for some $f_{(\bfc,t_\lz)}^{(\bfc',t_{\lz'})}\in v^{-1}\bbZ[v^{-1}]$.

 Thus
$$
E(\bfc,t_\lz)-\tilde{E}(\bfc,t_\lz)\in\mathbf{P}\cap\cC^*.
$$
By Lemma \ref{periodic not in C}, $$E(\bfc,t_\lz)=\tilde{E}(\bfc,t_\lz),$$
which proves the corollary.
\end{proof}

Now we give the main theorem of this paper, showing that the bar-invariant basis $\{C(\bfc,t_\lz)|(\bfc,t_\lz)\in\cG^a\}$ is the canonical basis $\bfB=\{B(\bfc,t_\lz)|(\bfc,t_\lz)\in\cG^a\}$ of $\bff$, and the bijection is clear.

\begin{theorem}
For all $(\bfc,t_\lz)\in\cG^a$, we have $C(\bfc,t_\lz)=B(\bfc,t_\lz)$.
\end{theorem}
\begin{proof}
By Section \ref{Section: construction of C(c,lz)} and Corollary \ref{Corollary: E in B}, we already have
$$C(\bfc,t_\lz)=B(\bfc,t_\lz)+\sum_{(\bfc',t_{\lz'})\in\cG^a\atop(\bfc',t_{\lz'})\prec(\bfc,t_\lz)}\tilde{b}_{(\bfc,t_\lz)}^{(\bfc',t_\lz')}(v)B(\bfc',t_{\lz'}),$$
for some $\tilde{b}_{(\bfc,t_\lz)}^{(\bfc',t_{\lz'})}\in v^{-1}\bbZ[v^{-1}]$.
Applying bar-involution to both sides, all $\tilde{b}_{(\bfc,t_\lz)}^{(\bfc',t_{\lz'})}=0$ and $B(\bfc,t_\lz)=C(\bfc,t_\lz)$ since all $B(\bfc,t_\lz)$ and $C(\bfc,t_\lz)$ are bar-invariant.
\end{proof}

Note that we not only prove that $$\{C(\bfc,t_\lz)|(\bfc,t_\lz)\in\cG^a\}=\{B(\bfc,t_\lz)|(\bfc,t_\lz)\in\cG^a\}$$ as sets, but also show that $$C(\bfc,t_\lz)=B(\bfc,t_\lz)$$ for any $(\bfc,t_\lz)\in\cG^a$.

By the almost orthogonality of $\{B(\bfc,t_\lz)|(\bfc,t_\lz)\in\cG^a\}$, the condition (2) of Theorem \ref{Theorem: C satisfy three conditions} follows.

Recall that in Lemma \ref{E in N}, for $(\bfc,t_\lz)\in\cG^a$, we have
$$E(\bfc,t_\lz)=N(\bfc,t_\lz)+\sum_{(\bfc',t_{\lz'})\in\cG_\nu\setminus\cG^a_\nu\atop (\bfc',t_{\lz'})\prec(\bfc,t_\lz)}b_{(\bfc,t_\lz)}^{(\bfc',t_{\lz'})}(v)N(\bfc',t_{\lz'}),$$
with $b_{(\bfc,t_\lz)}^{(\bfc',t_{\lz'})}\in\cZ$.
By Corollary \ref{Corollary: E in B}, $f_{(\bfc,t_\lz)}^{(\bfc',t_{\lz'})}=b_{(\bfc,t_\lz)}^{(\bfc',t_{\lz'})}$, so $b_{(\bfc,t_\lz)}^{(\bfc',t_{\lz'})}\in v^{-1}\bbZ[v^{-1}]$.
By the almost orthogonality of $\{N(\bfc,t_\lz)|(\bfc,t_\lz)\in\cG\}$, the PBW basis $\{E(\bfc,t_\lz)|(\bfc,t_\lz)\in\cG^a\}$ is almost orthogonal.
This completes the proof of Proposition \ref{proposition: E satisfy three conditions}.

\section{Another algorithm to compute the canonical basis}

In this section, we will give an easier algebraic algorithm to compute $C(\bfc,t_\lz)$, which relies on the proof of Theorem \ref{Theorem: C satisfy three conditions}.

Back to the construction of the monomial basis, we have for $(\bfc,t_\lz)\in\cG^a$, 
$$\fkm^{\oz(\bfc,t_\lz)}=N(\bfc,t_\lz)+\sum_{(\bfc',t_{\lz'})\in\cG\atop(\bfc',t_{\lz'})\prec(\bfc,t_\lz)}\phi_{(\bfc,t_\lz)}^{(\bfc',t_{\lz'})}(v)N(\bfc',t_{\lz'}),$$

For a polynomial $\phi(v)=\sum_{i\in\bbZ}\phi_i v^i\in\cZ$, let ${^+\phi}(v)=\phi_0+\sum_{i>0}\phi_i(v^i+v^{-i})$, then $\overline{^+\phi}={^+\phi}$ and $\phi-{^+\phi}\in v^{-1}\bbZ[v^{-1}]$.

Now fix $(\bfc,t_\lz)\in\cG^a_\nu$. For a maximal $(\bfc',t_{\lz'})\in\{(\bfc'',t_{\lz''})\in\cG^a_\nu|(\bfc'',t_{\lz''})\prec(\bfc,t_\lz)\}$, we have
\begin{eqnarray}
&&\fkm^{\oz(\bfc,t_\lz)}-{^+\phi}_{(\bfc,t_\lz)}^{(\bfc',t_{\lz'})}(v)\fkm^{\oz(\bfc',t_{\lz'})}\notag\\
&=&N(\bfc,t_\lz)+(\phi_{(\bfc,t_\lz)}^{(\bfc',t_{\lz'})}-{^+\phi}_{(\bfc,t_\lz)}^{(\bfc',t_{\lz'})})(v)N(\bfc',t_{\lz'})\notag\\
&&+\sum_{(\bfc'',t_{\lz''})\in\cG\atop(\bfc'',t_{\lz''})\prec(\bfc,t_\lz),(\bfc'',t_{\lz''})\neq(\bfc',t_{\lz'})}\varphi_{(\bfc,t_\lz)}^{(\bfc'',t_{\lz''})}(v)N(\bfc'',t_{\lz''}),\notag
\end{eqnarray}
where the left side is bar-invariant and the coefficient of $N(\bfc',t_{\lz'})$ in the right side is in $v^{-1}\bbZ[v^{-1}]$.

Whenever the right side has a term as the form $\tilde{\varphi}(v)N(\bfc',t_{\lz'})$ for $(\bfc',t_{\lz'})\in\cG^a$ with $(\bfc',t_{\lz'})\prec (\bfc,t_\lz)$ such that $\tilde{\varphi}(v)\not\in v^{-1}\bbZ[v^{-1}]$, let $(\bfc',t_{\lz'})$ be maximal with this property, we minus $^+\tilde{\varphi}(v)\fkm^{\oz(\bfc',t_{\lz'})}$ for both sides.
Continue this process until all coefficients of $N(\bfc',t_{\lz'}),(\bfc',t_{\lz'})\in\cG^a,(\bfc',t_{\lz'})\prec(\bfc,t_\lz)$ in the right side are in $v^{-1}\bbZ[v^{-1}]$.

In this way, we can get
\begin{align}\label{G(c,lz)}
&\fkm^{\oz(\bfc,t_\lz)}+\sum_{(\bfc',t_{\lz'})\in\cG^a\atop(\bfc',t_{\lz'})\prec(\bfc,t_\lz)}r_{(\bfc,t_\lz)}^{(\bfc',t_{\lz'})}(v)\fkm^{\oz(\bfc',t_{\lz'})}\notag\\
&=N(\bfc,t_\lz)+\sum_{(\bfc',t_{\lz'})\in\cG^a\atop(\bfc',t_{\lz'})\prec(\bfc,t_\lz)}s_{(\bfc,t_\lz)}^{(\bfc',t_{\lz'})}(v)N(\bfc',t_{\lz'})+\sum_{(\bfc',t_{\lz'})\in\cG\setminus\cG^a\atop(\bfc',t_{\lz'})\prec(\bfc,t_\lz)}\tilde{s}_{(\bfc,t_\lz)}^{(\bfc',t_{\lz'})}(v)N(\bfc',t_{\lz'})\notag
\end{align}
where all $r_{(\bfc,t_\lz)}^{(\bfc',t_{\lz'})}\in\cZ$ are bar-invariant and all $s_{(\bfc,t_\lz)}^{(\bfc',t_{\lz'})}$ are in $v^{-1}\bbZ[v^{-1}]$, while $\tilde{s}_{(\bfc,t_\lz)}^{(\bfc',t_{\lz'})}\in\cZ$ is not necessarily in $v^{-1}\bbZ[v^{-1}]$ for now (although later we can prove  that $\tilde{s}_{(\bfc,t_\lz)}^{(\bfc',t_{\lz'})}\in v^{-1}\bbZ[v^{-1}]$ by Proposition \ref{G=C} and almost-orthogonality).

Denote this element by
$$G(\bfc,t_\lz)=\fkm^{\oz(\bfc,t_\lz)}+\sum_{(\bfc',t_{\lz'})\in\cG^a\atop(\bfc',t_{\lz'})\prec(\bfc,t_\lz)}r_{(\bfc,t_\lz)}^{(\bfc',t_{\lz'})}(v)\fkm^{\oz(\bfc,t_\lz)}.$$

\begin{proposition}\label{G=C}
For all $(\bfc,t_\lz)\in\cG^a$, we have $G(\bfc,t_\lz)=C(\bfc,t_\lz)$.
\end{proposition}
\begin{proof}
Since $s_{(\bfc,t_\lz)}^{(\bfc',t_{\lz'})}\in v^{-1}\bbZ[v^{-1}]$, by doing the same operation in Definition \ref{def of E} on $G(\bfc,t_\lz)$, we can inductively get
$$
\hat{E}(\bfc,t_\lz)=G(\bfc,t_\lz)+\sum_{(\bfc',t_{\lz'})\in\cG^a\atop(\bfc',t_{\lz'})\prec(\bfc,t_\lz)}\gamma_{(\bfc,t_\lz)}^{(\bfc',t_{\lz'})}(v)G(\bfc',t_{\lz'})
$$
with $\gamma_{(\bfc,t_\lz)}^{(\bfc',t_{\lz'})}\in v^{-1}\bbZ[v^{-1}]$, such that $\tilde{E}(\bfc,t_\lz)$ is in $\cC^*$ for any $(\bfc,t_\lz)\in\cG^a$, and
$$
\hat{E}(\bfc,t_\lz)=N(\bfc,t_\lz)+\sum_{(\bfc',t_{\lz'})\in\cG\setminus\cG^a\atop(\bfc',t_{\lz'})\prec(\bfc,t_\lz)}\psi_{(\bfc,t_\lz)}^{(\bfc',t_{\lz'})}(v)N(\bfc',t_{\lz'})
$$
with $\psi_{(\bfc,t_\lz)}^{(\bfc',t_{\lz'})}\in\cZ$.

Hence, $\hat{E}(\bfc,t_\lz)-E(\bfc,t_\lz)\in\mathbf{P}\cap\cC^*$. By Lemma \ref{periodic not in C}, we have $\hat{E}(\bfc,t_\lz)=E(\bfc,t_\lz)$.
Recall that
$$C(\bfc,t_\lz)=E(\bfc,t_\lz)+\sum_{(\bfc',t_{\lz'})\in\cG^a\atop(\bfc',t_{\lz'})\prec(\bfc,t_\lz)}g^{(\bfc',t_{\lz'})}_{(\bfc,t_\lz)}(v)E(\bfc',t_{\lz'}),$$
with $g^{(\bfc',t_{\lz'})}_{(\bfc,t_\lz)}\in v^{-1}\bbZ[v^{-1}]$. Then
$$C(\bfc,t_\lz)=G(\bfc,t_\lz)+\sum_{(\bfc',t_{\lz'})\in\cG^a\atop(\bfc',t_{\lz'})\prec(\bfc,t_\lz)}\kappa^{(\bfc',t_{\lz'})}_{(\bfc,t_\lz)}(v)G(\bfc',t_{\lz'}),$$
with $\kappa^{(\bfc',t_{\lz'})}_{(\bfc,t_\lz)}\in v^{-1}\bbZ[v^{-1}]$.
Applying bar-involution to both sides, all $\kappa^{(\bfc',t_\lz')}_{(\bfc,t_\lz)}=0$ and $G(\bfc,t_\lz)=C(\bfc,t_\lz)$, since all $G(\bfc,t_\lz),C(\bfc,t_\lz)$ are bar-invariant.
\end{proof}

\bibliography{Tame_quivers_and_affine_bases}

\end{document}